\documentclass[11pt,reqno]{amsart}

\usepackage[hmargin=4cm,vmargin=4cm]{geometry}
\usepackage{amsmath,amssymb}
\usepackage{graphicx,subfigure}
\usepackage{mathtools}
\usepackage{thmtools,thm-restate}
\usepackage{color}
\usepackage{pinlabel}

\usepackage[normalem]{ulem}

\newtheorem{theorem}{Theorem}[section]
\newtheorem{proposition}[theorem]{Proposition}
\newtheorem{corollary}[theorem]{Corollary}
\newtheorem{conjecture}[theorem]{Conjecture}
\newtheorem{lemma}[theorem]{Lemma}

\newtheorem*{claim}{Claim}

\theoremstyle{definition}
 \newtheorem{definition}[theorem]{Definition}

\newtheorem{example}[theorem]{Example}

\newtheorem*{definition*}{Definition}

\theoremstyle{remark}
\newtheorem{remark}[theorem]{Remark}

\numberwithin{equation}{section}

\newcommand{\cU}{{\mathcal U}}

\def\RR{\mathbb{R}}

\renewcommand\SS{\mathbb{S}}
\def\HH{\mathbb{H}}

\newcommand{\cD}{{\mathcal D}}

\newcommand{\cG}{{\mathcal G}}

\newcommand{\cM}{{\mathcal M}}

\newcommand{\cS}{{\mathcal S}}

\newcommand{\cV}{{\mathcal V}}

\newcommand\minus\backslash

\newcommand\lan\langle
\newcommand\ran\rangle

\DeclareMathOperator\Div{div}

 \DeclareMathOperator\Int{int}

\newcommand\DD{\mathbb D}
\renewcommand\leq\leqslant
\renewcommand\geq\geqslant
\newlength{\intwidth}

\DeclareMathOperator\Imag{Im}

 \DeclareMathOperator\curl{curl}

\newcommand{\Rone}{\mathbb{R}}

\newcommand{\pp}[2]{\frac{\partial#1}{\partial#2}}

\author{Robert Cardona}
\address{Robert Cardona,
Departament de Matemàtiques i Informàtica, Universitat de Barcelona, Gran Via de les Corts Catalanes 585, 08007 Barcelona, Spain \newline \it{e-mail: robert.cardona@ub.edu}}

\author{Nathan Duignan}
\address{Nathan Duignan,
School of Mathematics and Statistics, University of Sydney, NSW 2050,
Australia \newline \it{e-mail: nathan.duignan@sydney.edu.au}}

\author{David Perrella}
\address{David Perrella,
The University of Western Australia, 35 Stirling Highway, Crawley WA 6009, Australia \newline \it{e-mail: david.perrella@uwa.edu.au}}

\thanks{RC acknowledges support from the Margarita Salas postdoctoral contract financed by the European Union-NextGenerationEU and its host institutions: the Universitat Politècnica de Catalunya and the Instituto de Ciencias Matemáticas. RC was partially supported by the AEI grant PID2019-103849GB-I00 / AEI / 10.13039/501100011033, the AGAUR grant 2021 SGR 00603, and the project Computational, dynamical and geometrical complexity in fluid dynamics - AYUDAS FUNDACIÓN BBVA A PROYECTOS INVESTIGACIÓN CIENTÍFICA 2021.}

\thanks{David Perrella would like to thank Daniel Peralta-Salas for providing financial support for his visit to the Instituto de Ciencias Matem{\'a}ticas in Madrid during which the idea for this paper came about. This paper was written while David Perrella received an Australian Government Research Training Program Scholarship at The University of Western Australia.}

\title{Asymmetry of MHD equilibria for generic adapted metrics}

\begin{document}

\begin{abstract}
Ideal magnetohydrodynamic (MHD) equilibria on a Riemannian 3-manifold satisfy the stationary Euler equations for ideal fluids. A stationary solution $X$ admits a large set of ``adapted" metrics in $M$ for which $X$ solves the corresponding MHD equilibrium equations with the same pressure function. We prove different versions of the following statement: an MHD equilibrium with non-constant pressure on a compact three-manifold with or without boundary admits no continuous Killing symmetries for an open and dense set of adapted metrics. This contrasts with the classical conjecture of Grad which loosely states that an MHD equilibrium on a toroidal Euclidean domain in $\RR^3$ with pressure function foliating the domain with nested toroidal surfaces must admit Euclidean symmetries.
\end{abstract}

\maketitle

\section{Introduction}

Ideal magneto-hydrodynamic equilibria on an oriented Riemannian 3-manifold $(M,g)$ possibly with boundary are described by solutions $(X,p)$ to the set of equations
\begin{equation}\label{eq:MHD}
\curl X \times X = \nabla p, \qquad \Div X = 0,
\end{equation}
where $X$ is called the magnetic field, $p$ the pressure function, and $\curl X$ the associated current density (up to units). If $\mu$ denotes the Riemannian volume-form of $(M,g)$, then for any vector fields $X$ and $Y$, the curl and cross products are uniquely defined by the equations
\begin{equation}\label{eq:curlandcrossprod}
i_{\curl X}\mu = d i_X g, \qquad i_{X \times Y} g = i_Yi_X\mu.
\end{equation}
In Equation \eqref{eq:curlandcrossprod}, $i$ denotes the interior product. Equation \eqref{eq:MHD} is equivalent to the Euler equations for a steady incompressible ideal fluid
\begin{equation}\label{steadyeulerfluid}
\nabla_X X + \nabla P = 0,  \quad \Div X = 0
\end{equation}
where $\nabla_X X$ denotes the covariant derivative of $X$ along itself. In particular, $(X,P)$ is a solution to Equation \eqref{steadyeulerfluid} if and only if $(X,p)$ solves Equation \eqref{eq:MHD} with $-p = \|X\|^2/2 + P$. In the fluid context, $p$ is known as the Bernoulli function.

A well-known and simple means of constructing solutions is under the assumption of axisymmetry. If $M$ is embedded in $\RR^3$ then a solution on $M$ is axisymmetric if there exists an axis in $\RR^3$ about which $X$, $p$, and $M$, have continuous rotational symmetry. Axisymmetric solutions to MHD equations are abundant and are constructed by solving an elliptic equation known as the Grad–Shafranov equation \cite{Ce}. Axisymmetric solutions may also be constructed with the desired property that the pressure function has toroidally nested level sets (see Definition \ref{def:toroial-lvl-sets} for a precise formulation of ``toroidally nested level sets"). 

Non-axisymmetric smooth solutions to Equation \eqref{eq:MHD} are also known to exist. For instance, Lortz in \cite{Lo} showed when $M$ is embedded in $\RR^3$ and has a reflection symmetry about a plane, then one can find MHD equilibria $(X,p)$ with $X$ and $p$ also possessing the reflection symmetry in a suitable sense with $X$ being tangent to the boundary. However, Lortz shows that the orbits of the MHD fields are all periodic in this situation and it appears unclear from the procedure that one may organise $p$ to have toroidally nested surfaces. In the setting of lower regularity, weak solutions $(X,p)$ to Equation \eqref{eq:MHD} with piece-wise constant pressure $p$, piece-wise smooth $X$, and tangential boundary conditions, are known to exist for non-axisymmetric domains. This was pioneered by Bruno and Laurence in \cite{BL} for small perturbations from axisymmetry and extended to a larger class of domains by Encisco, Luque, and Peralta-Salas in \cite{ELP}.

The lack of solutions to Equation \eqref{eq:MHD} with the pressure function having toroidally nested level sets led Grad in \cite{G} to make a loose conjecture about the nonexistence of solutions to Equation \eqref{eq:MHD} in contexts relevant to the design \cite{H} of magnetic confinement fusion devices. The conjecture loosely states, in the case of a solid torus $M$ embedded in $\RR^3$ (which we call a toroidal domain), that an MHD equilibrium $(X,p)$ on $M$ with $p$ having nested toroidal level sets, must possess either axisymmetry or a reflection symmetry. The authors in \cite{CDG2} summarise the full statement of the conjecture, including the situation of domains which are infinite cylinders. However, it is still an open problem \cite{ELP} to decide if there exists an MHD equilibrium $(X,p)$ on toroidal domain $M \subset \RR^3$ with the pressure function $p$ having toroidally nested level sets which is not axisymmetric. For this paper, we refer to this as Grad's conjecture (see Section \ref{sec:Grad conjecture} for a more precise discussion). 

Although Grad's conjecture remains unresolved, much is known about the necessary topology of MHD equilibria. Arnold's Structure Theorem \cite{A,AK} says that, when $M$ is compact and connected, the data in Equation \eqref{eq:MHD} is analytic, and $X$ and $\curl X$ are not everywhere colinear, then $M$ is, up to a set of measure zero, foliated by regular level sets of $p$ that are either cylinders or tori and $X$ is tangent to these level surfaces and diffeomorphic to a constant vector field on them. Arnold's proof also may be modified to hold for smooth $X$ and Morse-Bott $p$ \cite{Per}.

Besides the aforementioned examples of MHD equilibria on toroidal domains in $\RR^3$, examples can explicitly be constructed in the round three-sphere or in the flat three-torus, see for instance \cite{KKPS}. Furthermore, it was introduced in \cite{C} a method for constructing MHD equilibria with a Morse-Bott or analytic pressure function $p$ on any graph closed three-manifold with suitable adapted Riemannian metrics, providing a very large source of examples of such solutions. 

One can analyse the vector fields $X$ that satisfy Equation \eqref{eq:MHD} for some metric. In \cite{PRT} such vector fields are called Eulerizable, and a characterisation was given of non-vanishing Eulerizable flows on closed manifolds with vanishing first cohomology group. The characterisation is similar to that of Sullivan's homological characterisation of non-singular geodesible vector fields \cite{Sul}. An observation they make in \cite{PRT} which applies in general (that is, without assuming closedness of $M$ or that the first cohomology class of $M$ vanishes), is the following. For a vector field $X$, a 1-form $\alpha$, and a positively oriented volume form $\mu$ on $M$, consider the subset of the set $\cM(M)$ of all metrics on $M$ defined by
\begin{equation}\label{eq:adaptedsubspace}
\cM_{(X,\alpha,\mu)}(M) = \{g \in \cM(M) : i_X g = \alpha, \, \mu_g = \mu\}
\end{equation}
where, for $g \in \cM(M)$, $\mu_g$ denotes the Riemannian volume form on $(M,g)$. The elements $g \in \cM_{(X,\alpha,\mu)}(M)$ are said to be \emph{adapted to} $(X,\alpha,\mu)$. Supposing that
\begin{equation*}
L_X\mu = 0, \qquad i_Xd\alpha = -dp,
\end{equation*}
for some $p \in C^{\infty}(M)$, Equation \eqref{eq:curlandcrossprod} shows that $(X,p)$ is an MHD equilibrium for each metric $g$ adapted to $(X,\alpha,\mu)$. The notion of metric adapted to a one-form has its origins in contact geometry \cite{CH}, and was related to hydrodynamics in the seminal work of Etnyre and Ghrist \cite{EG}.

In this work, we use the formalism of adapted metrics to show that Grad's conjecture is radically false when generalised to the setting of smooth flows of isometries (as an analogy of axisymmetry) of general metrics on $M$. The nested structure of a solution may be interpreted as the pressure function $p$ being constant on the connected components of the boundary and regular almost everywhere, for instance, Morse-Bott or analytic (classical situations where Arnold's structure theorem applies). Even if the most interesting setting for Grad's conjecture is when $p$ has the aforementioned properties, it is enough in our theorems to assume that $p$ admits some regular level sets.

Assuming $M$ is compact, a smooth flow by isometries on $(M,g)$ is equivalent to the existence of a Killing field $K$ on $(M,g)$ which is tangent to $\partial M$, that is,
\begin{equation*}
K \cdot n = 0 \text{ on } \partial M,
\end{equation*}
where $n$ denotes the outward unit normal of $(M,g)$. Note that the tangency condition is vacuous when $\partial M = \emptyset$. Recall, a Killing field on $(M,g)$ is a vector field $K$ on $M$ satisfying
\begin{equation*}
L_K g = 0.
\end{equation*}

To state our main result, we will consider the set $\cM_{(X,\alpha,\mu)}(M)$ as defined in Equation \eqref{eq:adaptedsubspace} with the relative topologies induced by the (weak) $C^k$ topologies ($0 \leq k \leq \infty$) on $\cM(M)$. Throughout this paper, we use the weak $C^k$ topologies as opposed to the strong topologies for simplicity of exposition. Concerning generalised Grad's conjecture, we have the following result.

\begin{theorem}\label{thm:introalphafixed1}
Assume $M$ is compact and connected. Let $(X,p)$ be an MHD equilibrium on $(M,g_0)$ with pressure function $p$ such that $\nabla p \neq 0$ and $p$ is constant on $\partial M$. Set $\alpha = i_Xg_0$ and $\mu = \mu_{g_0}$, so that for any $g \in  \cM_{(X,\alpha,\mu)}(M)$, $(X,p)$ is an MHD equilibrium on $(M,g)$. Then, there is a $C^0$-open and $C^{\infty}$-dense subset $\cU \subset \cM_{(X,\alpha,\mu)}(M)$ such that, for each $g \in \cU$, there are no non-trivial Killing fields $K$ on $(M,g)$ satisfying $L_K X = 0$ or $K(p) = 0$.
\end{theorem}

An interesting feature of Theorem \ref{thm:introalphafixed1} (and Theorem \ref{thm:introalphafixed2} below) is that there is not much restriction on the topology of $p$ or the dynamics of $X$. In particular, one may take a familiar axisymmetric MHD equilibrium and perturb the metric to obtain metrics for which the conclusion of the theorem holds. We point out that in general there are metrics in the set $\cU$ that do admit Killing fields, even if the MHD solution does not (see Example \ref{eg:MHD-with-Killing-on-M} in Section \ref{sec:extensions}).

Theorem \ref{thm:introalphafixed1} is proven by showing that (1) the Killing field must preserve a scalar function on the regular toroidal level sets of $p$ which only depends on the metric in a $C^0$ way and (2) identifying a generic behavior of the scalar function which is incompatible with a non-vanishing Killing field and (3) that this behavior can be reached by perturbing in the class of adapted metrics. 

A natural question is whether a weaker form of Grad's conjecture for arbitrary metrics is also false. That is, whether there are no Killing fields on $M$, regardless of whether the Killing fields are symmetries of the given steady Euler flow. We will prove the following theorem regarding this.

\begin{theorem}\label{thm:introalphafixed2}
Assume $M$ is compact and $\partial M \neq \emptyset$. Let $(X,p)$ be an MHD equilibrium on $(M,g_0)$ with pressure function $p$ which is constant and regular on each of the connected components of $\partial M$. Set $\alpha = i_Xg_0$ and $\mu = \mu_{g_0}$, so that for any metric $g \in \cM_{(X,\alpha,\mu)}(M)$, $(X,p)$ is an MHD equilibrium on $(M,g)$. Then, there is a $C^2$-open and $C^{\infty}$-dense subset of $\cM_{(X,\alpha,\mu)}(M)$ which do not support non-trivial Killing fields $K$ tangent to $\partial M$.
\end{theorem}

Theorem \ref{thm:introalphafixed2} is proven similarly to Theorem \ref{thm:introalphafixed1}. The $C^2$-openness property comes about because we prove the result via perturbing the intrinsic scalar curvature on $\partial M$, inherited from $M$. It is natural to ask whether we can remove the boundary condition of the Killing field in Theorem \ref{thm:introalphafixed2}. Without this restriction, the geometry of Killing fields on $M$ as a whole must be considered in more detail. Nevertheless, these results show there exists a large class of metrics for which a generalised Grad's conjecture is false.

We will obtain Theorem \ref{thm:introalphafixed1} and \ref{thm:introalphafixed2} as corollaries of a more general result, namely Proposition \ref{prop:for-intro-alpha-fixed-1-and-2}. In particular, Proposition \ref{prop:for-intro-alpha-fixed-1-and-2} allows, for example, MHD equilibria with $\nabla p = 0$ which are known (here) as Beltrami fields. Explicitly, a Beltrami field in this paper refers to a vector field $X$ satisfying
\begin{equation*}
\Div X = 0, \qquad \curl X \times X = 0
\end{equation*}
on $(M,g)$. For one such result about Beltrami fields, see Corollary \ref{cor:beltrami-fields}. 

Lastly, it should be noted that Grad \cite{G} discussed the possibility that perhaps there do exist MHD equilibria on a toroidal domain with nested pressure surfaces which does not admit axisymmetry nor reflection symmetry but that such examples would be isolated in the sense that there is no curve (or family) of examples which pass through the given example (see also \cite{CDG2}). Grad's thoughts must be charitably interpreted here because one can do trivial things to create a family of examples, given a singular example, that a physicist would discount: such as rotating or applying a scaling to the example. For the interested reader, we provide some examples (see Example \ref{ex:families} in Section \ref{sec:extensions}) of a less trivial family with no symmetries (for some metrics that are adapted to the whole family) by applying our main results.\\

The structure of the paper is as follows. In Section \ref{sec:Grad conjecture}, we provide the reader with a mathematically precise discussion of Grad's conjecture focusing on the case of solid toroidal domains. In Section \ref{sec:pressure-and-Killing-fields}, we establish some geometric relationships between Killing fields and the pressure function of an MHD equilibrium. In Section \ref{sec:guided-flows}, we demonstrate that a class of volume-preserving vector fields with first integrals admit canonical symmetries. These symmetries are leveraged to prove Theorem \ref{thm:introalphafixed1}. In Section \ref{sec:perturbing-adapted-metrics}, a means of perturbing metrics in the class of adapted metrics is devised for the proof of Theorems \ref{thm:introalphafixed1} and \ref{thm:introalphafixed2}. In Section \ref{sec:extensions}, we apply the aforementioned generalisation of the main results to state some extensions of the main results and provide some examples related to plasma physics. In Section \ref{sec:discuss}, we discuss the results obtained in the paper in relationship how they may provide insight into resolving the classical conjecture of Grad.

\section{Some background on Grad's conjecture}\label{sec:Grad conjecture}

The purpose of this section is to precisely state the version of Grad's conjecture we focus on in this paper. First, the notion of a solid toroidal domain and the property of axisymmetry is defined, before a precise statement of Grad's conjecture is given in Conjecture \ref{conj:Grad}.

Let $M \subset \RR^3$ be a compact regular domain and $p \in C^{\infty}(M)$ be constant on the boundary. Then, any regular level set of $p$ is a closed (and orientable) 2-dimensional submanifold of $M$ (and thus of $\RR^3$). Moreover, for any closed connected 2-dimensional submanifold $S$ in $\RR^3$, by the Jordan-Brower Separation Theorem (see \cite{GP,Lima}), there exists a unique compact regular domain $R \subset \RR^3$ such that $S = \partial R$; we henceforth denote this domain by $R_S$. We make the following definition.

\begin{definition}\label{def:toroial-lvl-sets}
Define a \emph{solid toroidal domain} $M$ as a regular domain in $\RR^3$ diffeomorphic to $\DD^2 \times \SS^1$ where $\DD^2 \subset \RR^2$ is the closed unit disk and $\SS^1$ is the unit circle. If $M$ is a solid toroidal domain and $\gamma \subset \Int M$ is a 1-dimensional compact embedded submanifold, then a function $p \in C^{\infty}(M)$ is said to have \emph{toroidally nested level sets with axis $\gamma$} if each level set $S$ of $p$ is either
\begin{enumerate}
    \item a regular level set diffeomorphic to $\SS^1\times\SS^1$, or
    \item equal to $\gamma$,
\end{enumerate}
and moreover, for any two regular level sets $S_1,S_2$ of $p$, either $R_{S_1} \subset R_{S_2}$ or $R_{S_2} \subset R_{S_1}$. Lastly, for any regular level set $S$ of $p$, $\gamma \subset \Int R_S$ where $\gamma$ is the unique non-regular level set of $p$.
\end{definition}

The statement of Definition \ref{def:toroial-lvl-sets} contains some intentional redundancies which are intended to convey the geometric picture of nested toroidal level sets. The following elementary proposition removes some of the redundancies in the definition.

\begin{proposition}\label{prop:reducing-definition}
Let $M \subset \RR^3$ be a toroidal domain. Suppose that a function $p \in C^{\infty}(M)$ is constant on the boundary and the critical set $\gamma = \{x \in M : \nabla p|_x = 0\}$ is a connected embedded 1-dimensional submanifold in $\Int M$. Then, $p$ has toroidally nested level sets with axis $\gamma$.
\end{proposition}

\begin{proof}
First, note that the smooth function $p|_\gamma$ satisfies $d (p|_{\gamma}) = 0$ on $\gamma$ and because $\gamma$ is connected, $p$ is constant on $\gamma$. Because $p$ is a constant $p_0$ on the critical set of $M$ and a constant $p_1$ on $\partial M$, it follows that $p_0 \neq p_1$ and that one constant is the minimum of $p$ and the other is the maximum, both of which are not attained anywhere else. Without loss of generality, we assume that $p_0 < p_1$ so that $p_0$ is the minimum and $p_1$ is the maximum. 

As a standard application of the Flowout Theorem and Boundary Flowout Theorem \cite[Thm.~9.20, Thm.~9.24]{L} for a given closed 2-manifold $S$, the set
\begin{equation*}
\cD_S = \{x \in M \backslash \gamma : p^{-1}(p(x)) \text{ is diffeomorphic to $S$}\}   
\end{equation*}
is open in $M \backslash \gamma$. Because $M$ is connected, and $\gamma$ is compact, it is standard that $M \backslash \gamma$ is connected. Hence, because $\partial M$ is a 2-torus it follows from a simple connectedness argument that $\cD_{\SS^1\times\SS^1} = M \backslash \gamma$. Hence, every regular level set is a torus. 

Now, observe that $M = R_{\partial M}$ and if $S,S'$ are closed connected 2-dimensional submanifold $S$ in $\RR^3$, then a simple connectedness argument shows that if $S \subset R_{S'}$, then $R_S \subset R_{S'}$. With this, let $S$ be a regular level set of $p$. Then, $S \subset M$ so that by the above, $R_S \subset M$. Then, we may consider the restriction $p|_{R_S}$ which is constant on $\partial R_S = S$ and thus has a critical point in $\Int R_S$. Hence, $\gamma \cap \Int R_S \neq \emptyset$. On the other hand, $\gamma \cap S = \emptyset$ and thus $\gamma \cap (M \backslash \Int R_S)$ is open in $\gamma$. By connectedness of $\gamma$, we must have that $\gamma \subset \Int R_S$. In particular, denoting the constant value of $p$ on $S$ by $p_S$, this shows that
\begin{equation*}
R_S = p^{-1}((-\infty,p_S]).    
\end{equation*}
Thus, if $S_1$ and $S_2$ are regular level sets of $p$, with $p_{S_1} \leq p_{S_2}$, then $R_{S_1} \subset R_{S_2}$. From this, it is clear that $p$ has global toroidally nested level sets.
\end{proof}

For completeness, we formally define axisymmetry.

\begin{definition}
We say that an MHD equilibrium $(X,p)$ on a toroidal domain $M$ is \emph{axisymmetric} if there exists an infinitesimal generator $R$ of rotation about some axis in $\RR^3$ such that $R$ is tangent to $\partial M$ (so that the rotation preserves $M$) and the induced vector field $K$ on $M$ is a Killing symmetry of $X$ and $p$:
\begin{equation*}
L_{K} X = 0, \qquad L_{K} p = 0. 
\end{equation*}

\end{definition}

\begin{remark}
If $(X,p)$ is an MHD equilibrium on a toroidal domain $M$ and $K$ denotes a Killing field on $M$, if $p$ has toroidally nested level sets, then it follows from Lemma \ref{lem:tangenttobernoulli} and Proposition \ref{prop:KcommX} (see Section \ref{sec:pressure-and-Killing-fields}) that
\begin{equation*}
    L_K X = 0 \iff L_K p = 0.
\end{equation*}
\end{remark}

We now state a precise Grad's conjecture.

\begin{conjecture}\label{conj:Grad}
Let $M \subset \RR^3$ be a toroidal domain. Suppose there exists an MHD equilibrium $(X,p)$ on $M$ where $p$ has toroidally nested level sets. Then, the equilibrium is axisymmetric.
\end{conjecture}

\section{On the pressure function and Killing fields}\label{sec:pressure-and-Killing-fields}

In this section, we prove some important relationships between the Killing fields of a Riemannian 3-manifold $(M,g)$ and MHD equilibria $(X,p)$ on $(M,g)$. 

\begin{remark}\label{rem:killing-on-adapted-metric}
Let $g$ be a metric on $M$ adapted to $(X,\alpha,\mu)$ for some vector field $X$, 1-form $\alpha$, and volume form $\mu$. If $L_K X = 0$, where $K$ is a Killing field on $(M,g)$, then
\begin{equation*}
L_K \alpha = L_K i_X g = i_{L_K X} g + i_X L_K g = 0+0 = 0    
\end{equation*}
as well as
\begin{equation*}
L_K \mu = L_K \mu_g = 0.    
\end{equation*}
\end{remark}

The first property that we establish is that, under mild assumptions, the Killing fields that are symmetries of $X$ are necessarily tangent to the level sets of the pressure function $p$. More precisely, the following lemma holds.

\begin{lemma}\label{lem:tangenttobernoulli}
Let $(X,p)$ be an MHD equilibrium on $(M,g)$ and $K$ be a Killing field such that $L_K X = 0$. Then $L_K \alpha = 0$, where $\alpha = i_X g$, and $dL_Kp = L_Kdp = 0$. In particular, if $M$ is connected and $p$ has a critical point, then $L_K p = 0$.
\end{lemma}

\begin{proof}
Remark \ref{rem:killing-on-adapted-metric} gives immediately that $L_K \alpha = 0$. The fact that $\alpha$ is preserved by $K$ follows immediately from Remark \ref{rem:killing-on-adapted-metric}. On the other hand
\begin{equation*}
-dp = i_X d\alpha,  
\end{equation*}
hence
\begin{align*}
-L_K dp = L_K(i_X d\alpha) = i_{(L_K X)} d\alpha + i_X L_Kd\alpha = i_X dL_K\alpha = 0.
\end{align*}
That is, $L_K dp = 0$. We of course have $L_K dp = d L_K p$ and so
\begin{equation*}
L_K p = dp(K)    
\end{equation*}
is a function whose exterior derivative is zero. Moreover, at critical points of $p$, $L_K p$ necessarily vanishes. The conclusion immediately follows.
\end{proof}

For the next statement, given a smooth function $p \in C^{\infty}(M)$, we set
\begin{equation*}
\text{Reg}_M(p) = \{x \in M : dp|_x \neq 0\}.
\end{equation*}
The proposition below establishes a partial converse to Lemma \ref{lem:tangenttobernoulli}. 

\begin{proposition}\label{prop:KcommX}
Let $X$ be a vector field and $p \in C^{\infty}(M)$ have compact level sets. Assume that, with respect to the metric $g$,
\begin{equation}\label{eq:nearly-mhd}
X \cdot \nabla p = 0 = \curl X \cdot \nabla p \text{ on } M, \qquad \Div X = 0 \text{ on } M, \qquad X \cdot n = 0 \text{ on } \partial M.
\end{equation}
\begin{enumerate}
    \item Suppose that $K$ is a Killing field on $(M,g)$ tangent to $\partial M$ such that
\begin{equation*}
L_K p = 0.
\end{equation*}
Then,
\begin{equation*}
L_K X|_{\text{Reg}_M(p)} = 0.
\end{equation*}
\item Let $S$ be a surface embedded in $M$ on which $p$ is constant and regular and $X$ is tangent to $\partial S$. If $X|_S \neq 0$ but vanishes somewhere on $S$ and $K$ is tangent to $\partial S$, then $K = 0$.
\end{enumerate}
\end{proposition}

Proposition \ref{prop:KcommX} will be proven in this section shortly after establishing some more theory which is also used throughout the paper. Note that the proposition accounts for an MHD equilibrium $(X,p)$ because, as seen immediately from Equation \eqref{eq:MHD}, $p$ is a first integral of both $X$ and $\curl X$. We consider Equation \eqref{eq:nearly-mhd} (as opposed to the less general MHD equations) primarily because it additionally captures first integrals that may exist for Beltrami fields (see the forthcoming Remark \ref{rmk:Beltrami}). Recall we have defined these as steady Euler flows on $(M,g)$ (where $g$ is a chosen metric on $M$) such that the pressure function satisfies $\nabla p = 0$. That is, a vector field $X$ on $M$ such that
\begin{equation*}
\Div X = 0, \qquad \curl X \times X = 0.
\end{equation*}
We discuss some more the notion of Beltrami field.

\begin{remark}\label{rmk:Beltrami}
If $X$ is a Beltrami field on $(M,g)$, then, for some (not necessarily smooth) function $\lambda : M \to \RR$,  
\begin{equation}\label{eq:Beltrami-proportionality}
\curl X = \lambda X.
\end{equation}
A function $\lambda : M \to \RR$ in Equation \ref{eq:Beltrami-proportionality} is called a proportionality factor of $X$ (it is not unique if, for instance, $X$ has zeros and continuity of the proportionality factor is not required). 

In particular, if $p$ is a function on $(M,g)$ such that $X \cdot \nabla p = 0$, then also $\curl X \cdot \nabla p = 0$ which means that, with boundary conditions imposed, Lemma \ref{prop:KcommX} applies in this setting.

To prove the existence of a proportionality factor as in Equation \eqref{eq:Beltrami-proportionality}, we note that
\begin{equation*}
0 = \Div(X \times \curl X) = \curl X \cdot \curl X - X \cdot \curl \curl X    
\end{equation*}
and therefore, for $x \in M$, if $X|_x = 0$, then $\curl X|_x = 0$ and thus, Equation \eqref{eq:Beltrami-proportionality} holds. 

If $X$ has a smooth proportionality factor $\lambda$ (such is the case when $X$ has no zeros) then
\begin{equation*}
0 = \Div \curl X = \Div \lambda X = \nabla \lambda \cdot X + \lambda \Div X = \nabla \lambda \cdot X   
\end{equation*}
so that $\lambda$ is a first integral of $X$. 

A result similar to Lemma \ref{lem:tangenttobernoulli} holds for smooth proportionality factors $\lambda$ (instead of $p$) but without any global restrictions on $\lambda$. The global restriction placed on \ref{lem:tangenttobernoulli} was that $p$ was assumed to have a critical point. To show this similar result, writing $\alpha = i_X g$, we have
\begin{equation*}
d\alpha = i_{\curl X} \mu = \lambda i_X \mu.    
\end{equation*}
So, if $K$ is a Killing field on $(M,g)$ with $L_K X = 0$. We obtain from Lemma \ref{lem:tangenttobernoulli} that $L_K \alpha = 0$ and therefore
\begin{equation*}
0 = L_K d\alpha = L_K \lambda i_X \mu = (L_K\lambda)i_X \mu + \lambda L_K(i_X \mu) = (L_K\lambda)i_X \mu
\end{equation*}
so that $L_K \lambda = 0$.
\end{remark}

To prove Proposition \ref{prop:KcommX}, we introduce the notion $P$-harmonic 1-forms. This follows the work in \cite{PPS} but in the case of manifolds with boundary. We emphasise to the reader at this point that in many of our arguments, we consider the case of manifolds with boundary with the understanding that the case of empty boundary makes some steps of proofs vacuously true by virtue of there being no boundary points to consider.

\begin{definition}
Let $S$ be a compact oriented 2-manifold with (possibly empty) boundary, henceforth called a surface, and $P \in C^{\infty}(S)$ be a positive function. Let $h$ be a metric on $S$ (henceforth, $(S,h)$ is said to be a Riemannian surface). We say that a 1-form $\omega \in \Omega^1(S)$ is a $P$-harmonic 1-form on $(S,h)$ if
\begin{equation*}
d\omega = 0, \qquad \delta P \omega = 0
\end{equation*}
(where $\delta$ is the codifferential on $S$) and we say in addition that $\omega$ is Neumann if
\begin{equation*}
\omega(n) = 0   
\end{equation*}
where $n$ denotes the outward unit normal of $(S,h)$.
\end{definition}

The notion of $P$-harmonic forms enters in the proof of Proposition \ref{prop:KcommX} through the following lemma.

\begin{lemma}\label{lemma:XisP-harmonic}
Let $p \in C^{\infty}(M)$ and $X$ a vector field on $M$ satisfying
\begin{equation}\label{psuedoMHS}
X \cdot \nabla p = 0 = \curl X \cdot \nabla p \text{ on } M, \qquad \Div X = 0 \text{ on } M, \qquad X \cdot n = 0 \text{ on } \partial M
\end{equation}
with respect to the metric $g$, where $n$ is the outward unit normal. Let $S$ be a surface embedded in $M$ such that $p$ is constant and regular on $S$ and $X$ is tangent to $\partial S$. Then, the 1-form
\begin{equation*}
\omega = i^*X^{\flat},
\end{equation*}
where $X^{\flat} = i_X g$, is a Neumann $P$-harmonic 1-form on $(S,i^*g)$, where $S$ is oriented by $M$ and $\nabla p$, whereby
\begin{equation*}
P = \frac{1}{\|\nabla p\||_S}.
\end{equation*}
Moreover, if $X|_S \neq 0$, then $H_{\text{dR}}^{1}(S) \neq \{0\}$, where $H_{\text{dR}}^1(S)$ denotes the first de Rham cohomology of $S$.
\end{lemma}

\begin{proof}
Although this is implicitly proven in \cite{PPS} for the case of $\partial S = \emptyset$, for completeness, we give here a direct proof in the general case. Firstly, by restricting to open subsets, note that it suffices to consider the case where $p$ is regular everywhere, although we will only leverage this when computing $\delta P\omega$. To compute $d\omega$, first note that
\begin{equation*}
d\omega = i^*dX^{\flat} = i^*(i_{\curl X} \mu).   
\end{equation*}
Next, we have as well
\begin{equation*}
(i_{\curl X} \mu) \wedge dp = dp(\curl X) \mu = (\nabla p \cdot \curl X)\mu = 0.   
\end{equation*}
Then, because $i^*dp = 0$ and $dp|_S$ has no zeros, evaluating $(i_{\curl X} \mu) \wedge dp$ at points $x \in S$ along vectors tangent to $S$ shows that $i^*(i_{\curl X} \mu) = 0$ and thus
\begin{equation*}
d\omega = 0.    
\end{equation*}

To compute $\delta P\omega$, consider the vector field $X_S$ related to $X$ on $S$ via the inclusion $i : S \subset M$ (that is, $Ti \circ X_S = X \circ i$). Then if $\mu_S$ denotes the Riemannian volume form on $(S,i^*g)$, we have
\begin{equation*}
(\delta P\omega)\mu_S = -L_{P X_S}\mu_S = -L_{X_S}(P \mu_S)    .
\end{equation*}
Because $p$ is constant and regular on $S$, a unit normal section $\tilde{n} : S \to TM$ for $S$ in $M$ is given by the formula $\tilde{n} = \nabla p/\|\nabla p\| |_S$ and thus the volume form $\mu_S$ may be written as
\begin{equation*}
\mu_S = i^*(i_{\nabla p/\|\nabla p\|}\mu)   
\end{equation*}
and by naturality of Lie derivatives, setting $\tilde{P} = \frac{1}{\|\nabla p\|}$, we have 
\begin{equation*}
L_{X_S}P \mu_S = i^*(L_X \tilde{P} i_{\nabla p/\|\nabla p\|}\mu) = i^*(L_X i_{N}\mu),
\end{equation*}
where $N = \nabla p/\|\nabla p\|^2$. On the other hand
\begin{equation*}
i_N \mu \wedge dp = dp(N) \mu = \mu,
\end{equation*}
and thus, because $L_X dp = 0$ and $L_X \mu = 0$, we deduce that
\begin{equation*}
0 = L_X\mu = L_X(i_N \mu) \wedge dp + i_N \mu \wedge L_X dp = L_X(i_N \mu) \wedge dp.    
\end{equation*}
As before, because $i^*dp = 0$ and $dp|_S$ has no zeros, we obtain that $i^*L_X(i_N \mu) = 0$ and therefore
\begin{equation*}
L_{X_S}P \mu_S = 0.
\end{equation*}
This implies that
\begin{equation*}
\delta P\omega = 0,    
\end{equation*}
and because $X$ is tangent to $\partial S$, it implies $X_S$ is tangent to $\partial S$ and therefore $\omega$ satisfies
\begin{equation*}
\omega(n_S) = 0,     
\end{equation*}
where $n_S : \partial S \to TS$ is the outward unit normal of $(S,i^*g)$. Thus, $\omega$ is a $P$-harmonic 1-form. 

Now suppose that $X|_S \neq 0$. Then $\omega \neq 0$ because $X$ is tangent to $S$. Following Remark \ref{rmk:exact-neumann-harmonic-form}, we must have that the cohomology class $[\omega] \in H_{\text{dR}}^1(S)$ is non-zero. Thus, $H_{\text{dR}}^{1}(S) \neq \{0\}$.
\end{proof}

We now develop some basic properties of $P$-harmonic 1-forms.

\begin{remark}\label{rmk:exact-neumann-harmonic-form}
When $\partial S = \emptyset$, the Neumann boundary condition is vacuous. Neumann $P$-harmonic 1-forms bear many similarities with Neumann harmonic 1-forms (for the case of $\partial S = \emptyset$, see \cite{W}). For instance, all exact Neumann $P$-harmonic 1-forms are identically zero. Indeed, suppose that $\omega = df$ is a Neumann $P$-harmonic 1-form which is also exact. Then by Green's formula, letting $n$ and $t$ respectively denote the normal and tangential component operators (see \cite{S} for details), we obtain that
\begin{equation*}
(\omega,P \omega)_{L^2} = (df,P \omega)_{L^2} = (f,\delta P \omega)_{L^2} + \int_{\partial S} tdf \wedge \star n\omega = 0
\end{equation*}
so that $\omega = 0$. That is, all exact Neumann $P$-harmonic 1-forms are zero. 
\end{remark}

Given the similarity between harmonic 1-forms and $P$-harmonic 1-forms, it is unsurprising that $P$-harmonic 1-forms bear the following relationship with Killing fields.

\begin{lemma}\label{lem:K-pres-harmonic}
Let $K$ be a Killing field on a Riemannian surface $S$ which is tangent to the boundary and satisfies $L_K P = 0$. If $\omega \in \Omega^1(S)$ is a Neumann $P$-harmonic 1-form on $S$, then
\begin{equation*}
L_K\omega = 0.
\end{equation*}
\end{lemma}

\begin{proof}
Since $K$ is tangent to the boundary and $S$ is compact, $K$ is a complete vector field on $S$. Therefore, there exists a global flow $\psi^K$ of $K$ for which $\partial S$ is invariant. Letting $t \in \mathbb{R}$ and using the fact that $\psi^K_t$ is an orientation-preserving isometry, we have that
\begin{align*}
\left(\psi^K_t\right)^* \delta &= \delta \left(\psi^K_t \right)^*\\
n \left(\psi^K_t(x)\right) &= T\psi^k_t|_x n(x)
\end{align*}
for $x \in \partial S$. Now, $\psi_t^K$ is isotopic to the identity for any $t$ and hence induces the identity map in cohomology. By assumption, we also have that $\left(\psi^K_t\right)^*P = P$. Altogether, this demonstrates that if $\omega $ is a Neumann $P$-harmonic 1-form, then $\left(\psi_t^K\right)^*\omega$ is a Neumann $P$-harmonic 1-form also and their difference is exact; namely
\begin{equation*}
\left(\psi_t^K \right)^*\omega - \omega = df
\end{equation*}
for some $f \in C^{\infty}(S)$. However, this in turn implies $df$ is an exact Neumann $P$-harmonic 1-form and therefore zero by Remark \ref{rmk:exact-neumann-harmonic-form}.
\end{proof}

We mention here an important lemma for the paper which also plays a role in the proof of Proposition \ref{prop:KcommX}.

\begin{restatable}{lemma}{vanishingonsurface}\label{lem:vanishingonsurface}
Let $g$ be a metric on a connected manifold $M$ with boundary and let $K$ a Killing field on $(M,g)$. Let $S$ be a submanifold of $M$ with boundary. Then, the following hold.
\begin{enumerate}
    \item If $K$ is tangent to $S$, then the vector field $K_S$ related to $K$ via the inclusion $i : S \subset M$, is a Killing field on the Riemannian manifold $(S,i^*g)$ with boundary.
    \item If $K|_S = 0$ and $S$ has codimension $1$, then $K = 0$ on $M$.
\end{enumerate}
\end{restatable}

Lemma \ref{lem:vanishingonsurface} is an elementary fact proven in Appendix \ref{app:surfacegeo}. Only part (i) enters Proposition \ref{prop:KcommX}. Part (ii) enters the proof of the main result of the paper as a mechanism for ruling out the existence of Killing fields on a 3-manifold with boundary; once it is known that a Killing field vanishes on a surface, according to (ii) it follows that the Killing field vanishes everywhere. This is used in conjunction with the following Lemma.

\begin{lemma}\label{lem:vanishinginsurface}
Let $S$ be a connected Riemannian surface such that $\text{H}_{dR}^1(S) \neq \{0\}$. If $K$ is a Killing field on $S$ which is tangent to $\partial S$, then $K = 0$ or $K$ has no zeros.
\end{lemma}

\begin{proof}
Because $H^1_{\text{dR}}(S) \neq \{0\}$, there exists a Neumann harmonic 1-form $\omega \neq 0$ (see \cite[Thm.~2.6.1]{S}). Thus, $L_K \omega = 0$ by Lemma \ref{lem:K-pres-harmonic} (with $P = 1$). Now we may consider also $\star \omega$ and obtain also that $L_K{\star \omega} = 0$. We have
\begin{align*}
0 = L_K \omega = i_Kd\omega + d i_K\omega = d\omega(K)
\end{align*}
and analogously $d{\star \omega}(K) = 0$. Hence, both $\omega(K)$ are ${\star \omega}(K)$ are constant. Now, taking a point $x \in S$ where $\omega|_x \neq 0$, we have that $\omega|_x,\star\omega|_x$ is a basis for $T^*_xS$ at $x$ and therefore, if $K \neq 0$, then either $\omega(K)|_x$ or ${\star \omega}(K)|_x$ is non-zero. Thus, if $K \neq 0$, then $K$ has no zeros.
\end{proof}

\begin{remark}
In the proof of Lemma \ref{lem:vanishinginsurface}, it was found that either $\omega(K)$ is nowhere-vanishing or ${\star \omega}(K)$ is nowhere-vanishing. By the Neumann boundary condition on $\omega$, the 1-form ${\star \omega}$ satisfies $\jmath^*{\star \omega} = 0$ where $\jmath : \partial S \subset S$ is the boundary inclusion. Thus, ${\star \omega}(K)|_x = 0$ for $x \in \partial S$. Thus, when $\partial S \neq \emptyset$, we must have ${\star \omega}(K) = 0$ and that $\omega(K)$ has no zeros.
\end{remark}

To prove Proposition \ref{prop:KcommX}, we recall a modified notion of a regular surface for manifolds with boundary; similar to that considered by Arnold in \cite[Thm.~1.2.]{AK}. This is treated in more detail in Appendix \ref{app:partialregvalues}.

Let $p \in C^{\infty}(M)$. Then, we call $y \in \mathbb{R}$ a $\partial$-regular value of $p$ if $y$ is a regular value of $p$ and $p|_{\partial M}$ and we call $p^{-1}(y)$ a $\partial$-regular level set of $p$. The $\partial$-regular level sets $S$ of $p$ are surfaces embedded in $M$ such that
\begin{equation*}
\partial S = S \cap \partial M. 
\end{equation*}
See Theorem \ref{thm:preimagepartial} in Appendix \ref{app:partialregvalues}. The remaining fact we will use about the $\partial$-regular surfaces of a smooth $p : M \to \mathbb{R}$ is their density. Setting
\begin{align*}
\text{Reg}^\partial (p) &= \{y \in \mathbb{R} : y \text{ is a $\partial$-regular value of } p\},
\end{align*}
the fact we will use is that $p^{-1}(\text{Reg}^\partial (p))$ is a dense subset of the open set $\text{Reg}_M(p)$. This is a well-known application of Sard's Theorem (see Corollary \ref{cor:partialdensity} in Appendix \ref{app:partialregvalues}). 

We now give a proof of Proposition \ref{prop:KcommX}.

\begin{proof}[Proof of Proposition \ref{prop:KcommX}]

We prove part (ii) of the proposition first and use elements of the proof for part (i). For any such surface $S$ in the statement of the proposition, as per Lemma \ref{lemma:XisP-harmonic}, writing $i : S \subset M$ for the inclusion, $\omega = i^*X^{\flat}$ is a $P$-harmonic 1-form where
\begin{equation*}
P = \frac{1}{\|\nabla p\||_S}.    
\end{equation*}
Now, we have that $L_K dp = 0$ and thus $L_K \nabla p = 0$ (as $i_{\nabla p}g = dp$) so that
\begin{equation}
L_K \|\nabla p\|^2 = 0.      
\end{equation}
Hence, considering the vector field $K_S=i^*K$ (which exists as $K$ is tangent to $S$), because $K_S$ is a Killing field on $(S,i^*g)$, by Lemma \ref{lem:K-pres-harmonic} we have 
\begin{equation}\label{eq:w-pres-by-Ks}
L_{K_S}\omega = 0.    
\end{equation}
Now, supposing $X|_S \neq 0$, we have from Lemma \ref{lemma:XisP-harmonic} that $H_{\text{dR}}^1(S) \neq 0$. In particular, either $K_S = 0$ or $K_S$ has no zeros. Because $\omega$ is closed, we have by Equation \eqref{eq:w-pres-by-Ks} that
\begin{equation*}
d\omega(K_S) = 0.   
\end{equation*}
Moreover, ${\star P} \omega$ is closed as $\omega$ is $P$-harmonic. Thus, by Equation \eqref{eq:w-pres-by-Ks}, and the fact that $K_S$ is Killing, we get
\begin{equation*}
d{\star P}\omega(K_S) = L_{K_S} {\star P}\omega = 0.   
\end{equation*}
Hence, $\omega(K_S)$ and ${\star P}\omega(K_S)$ are constants. On the other hand, $\omega$ vanishes at some point (since $X$ does) and thus
\begin{equation*}
\omega(K_S) = 0 =({\star P}\omega)(K_S).
\end{equation*}
However, then taking some point $x \in S$ where $\omega|_x \neq 0$, we obtain that $K_S|_x = 0$ and thus $K_S = 0$ by Lemma \ref{lem:vanishinginsurface}. Hence, $K|_S = 0$ and so by Lemma \ref{lem:vanishingonsurface}, $K = 0$.

For part (i), let $S$ be a connected component of a $\partial$-regular level set of $p$. Then, for any vector field $Y$ on $M$ tangent to $\partial M$ and $S$, because $\partial S = S \cap \partial M$ and $S$ and $\partial M$ have transverse intersection at boundary points $x \in \partial S$, it follows that $Y$ is tangent to $\partial S$ too. Hence, it follows that $X$ and $K$ are tangent to $S$ and $\partial S$. Thus, Equation \eqref{eq:w-pres-by-Ks} applies and because $K_S$ is Killing,
\begin{equation*}
L_{K_S} X_S = 0    
\end{equation*}
where $X_S$ the the vector field related to $X$ on $S$ via the inclusion. It follows that
\begin{equation*}
L_K X |_S = 0.    
\end{equation*}
However, we recall that $\partial$-regular level sets of $p$ are dense in $\text{Reg}_M(p)$. It follows that
\begin{equation*}
 L_K X |_{\text{Reg}_M(p)} = 0.  
\end{equation*}
\end{proof}

\section{Guided flows}\label{sec:guided-flows}

Throughout this section, we fix $M$ to be an oriented 3-manifold with (possibly empty) boundary. Here we adopt a similar formalism as in \cite{PRT} but for a slightly more general class of flows.

\begin{definition}\label{def:guidedflow}
A quadruple $(X,\alpha,\mu,p)$ with $X$ a vector field, $\alpha$ a 1-form, $\mu$ a volume form, and $p$ a function on $M$, is said to be a \emph{guided flow} if
\begin{equation*}
\alpha(X) > 0, \qquad L_X \mu = 0, \qquad dp(X) = 0, \qquad d\alpha \wedge dp = 0.
\end{equation*}
\end{definition}

\begin{remark}\label{rmk:guided-is-flux}
In the terminology of \cite{PDP}, a triple $(X,\mu,\nu)$ with $X$ a vector field, $\mu$ a volume form, and $\nu$ a closed 1-form on $M$ satisfying
\begin{equation*}
L_X \mu = 0, \qquad \nu(X) = 0
\end{equation*}
is called a flux system. A 1-form $\alpha$ on $M$ is said to be adapted to the flux system $(X,\mu,\nu)$ if
\begin{equation*}
\alpha(X) > 0, \qquad d\alpha \wedge \nu = 0.
\end{equation*}
From this point of view, a guided flow $(X,\alpha,\mu,p)$ is a flux system $(X,\mu,dp)$ equipped with an adapted 1-form $\alpha$.
\end{remark}

The main application towards ideal MHD equilibria is facilitated by the following.

\begin{proposition}\label{prop:guidedflows}
Let $X$ be a nowhere-vanishing vector field on $M$ and $p \in C^{\infty}(M)$ such that, for some metric $g$ on $M$,
\begin{equation*}
\Div X = 0, \qquad X \cdot \nabla p = 0, \qquad \curl X \cdot \nabla p = 0. 
\end{equation*}
Then $(X,i_Xg,\mu,p)$ is a guided flow on $M$.
\end{proposition}

\begin{remark}
Proposition \ref{prop:guidedflows} in particular says that if $(X,p)$ is an MHD equilibrium on $(M,g)$ without zeros, then $(X,i_X g,\mu,p)$ is a guided flow. If $dp = 0$ so that in fact $X$ is a Beltrami field, then
\begin{equation*}
\curl X = \lambda X    
\end{equation*}
for some smooth $\lambda$ and, as in Remark \ref{rmk:Beltrami}, $\lambda$ is a first integral of $X$ (and a first integral of $\curl X$) making $(X,i_X g,\mu,\lambda)$ a guided flow. If $X$ has some first integral $f$, then again by Remark \ref{rmk:Beltrami}, $(X,i_X g,\mu,f)$ is a guided flow.
\end{remark}

The structure of a guided flow induces an important additional vector field, which we call the companion vector field.

\begin{theorem}[{\cite[Thm.~I.3]{PDP}}]\label{thm:prev-paper}
Given a guided flow $(X,\alpha,\mu,p)$, consider the vector field $Y$ such that
\begin{equation*}
i_Y \mu = \frac{1}{\alpha(X)}\alpha \wedge dp,
\end{equation*}
which we will call the \emph{companion vector field of the guided flow $(X,\alpha,\mu,p)$}. Then,
\begin{equation*}
\alpha(Y) = 0, \qquad dp(Y) = 0.
\end{equation*}
Moreover, setting $\tilde{X} = \frac{1}{\alpha(X)}X$, we have
\begin{equation*}
[\tilde{X},Y] = 0.
\end{equation*}
\end{theorem} 

The vector field $\tilde{X}$ and $Y$ in Theorem \ref{thm:prev-paper} form a nice frame for an adapted metric on constant-$p$ surfaces as follows.

\begin{proposition}\label{prop:companionflow}
Given a guided flow $(X,\alpha,\mu,p)$, let $Y$ be the companion field. If $S$ is a surface on which $p$ is constant and regular, then $\tilde{X}$ and $Y$ descend to everywhere linearly independent commuting vector fields on $S$ which we denote by $\tilde{X}_S$ and $Y_S$. If $g$ is a metric on $M$ adapted to $(X,\alpha,\mu)$ and $i : S \subset M$ denotes the inclusion, the metric $i^*g$ on $S$ may be written as
\begin{equation*}
    i^*g = E \omega^2 + G \eta^2 
\end{equation*}
where
\begin{equation*}
E = 1/\alpha(X), \qquad G = g(Y,Y)|_S    
\end{equation*}
and $(\omega,\eta)$ are the dual co-frame of $(\tilde{X}_S,Y_S)$. In particular, $\omega$ and $\eta$ are closed 1-forms on $S$ and $\omega = i^*\alpha$.   
\end{proposition}

\begin{proof}
Proven in \cite[Thm.~I.3]{PDP} (see Remark \ref{rmk:guided-is-flux}) are the facts that
\begin{equation*}
\alpha(Y) = 0, \qquad dp(Y) = 0, \qquad [\tilde{X},Y] = 0.
\end{equation*}
In particular, $\tilde{X}$ and $Y$ are vector fields tangent to $S$ and thus we may consider the related vector fields $\tilde{X}_S$ and $Y_S$ on $S$. Concerning the metric $i^*g$ on $S$, for any $Z \in \{\tilde{X},Y\}$, denoting by $Z_S$ the corresponding vector field on $S$, we have
\begin{align*}
(i^*g)(\tilde{X}_S,Z_S) &= g(\tilde{X},Z)|_S\\
&= \frac{1}{\alpha(X)}g(X,Z)|_S\\
&= \frac{1}{\alpha(X)}\alpha(Z)|_S.
\end{align*}
We deduce that
\begin{equation*}
(i^*g)(\tilde{X}_S,\tilde{X}_S) = 1/\alpha(X), \qquad (i^*g)(\tilde{X}_S,Y_S) = 0,
\end{equation*}
and hence we have
\begin{equation*}
(i^*g)(Y_S,Y_S) = g(Y,Y)|_S.
\end{equation*}
Now consider the dual co-frame $(\omega,\eta)$ of $(\tilde{X}_S,Y_S)$. We observe that
\begin{align*}
d\omega(\tilde{X}_S,Y_S) &= \tilde{X}_S(\omega(Y_S)) - Y_S(\omega(\tilde{X}_S)) - \omega([\tilde{X}_S,Y_S])\\
&= \tilde{X}_S(0) - Y_S(1) - \omega(0)\\
&= 0
\end{align*}
and hence $d\omega = 0$. Similarly one concludes that $d\eta = 0$. Lastly, observe that
\begin{equation*}
(i^*\alpha)(\tilde{X}_S) = \alpha(\tilde{X})|_S = 1,\qquad
(i^*\alpha)(Y_S) = \alpha(Y)|_S = 0
\end{equation*}
which proves that $\omega = i^*\alpha$.
\end{proof}

A vector field preserving the structure of a guided flow will preserve the companion vector field, as seen in the following.

\begin{lemma}\label{lem:companionispreserved}
Let $(X,\alpha,\mu,p)$ be a guided flow and let $K$ be a vector field on $M$ which is a symmetry of $X$, $\alpha$, $\mu$ and $p$. Then,
\begin{equation*}
L_K Y = 0    
\end{equation*}
where $Y$ is the companion to $(X,\alpha,\mu,p)$.
\end{lemma}

\begin{proof}
We have
\begin{equation*}
i_Y\mu = \frac{1}{\alpha(X)}\alpha \wedge dp.
\end{equation*}
and so
\begin{equation}\label{eq:applyingK}
L_K(i_Y\mu) = L_K\left(\frac{1}{\alpha(X)}\alpha \wedge dp\right).
\end{equation}
On the left-hand side, we have
\begin{equation}\label{eq:lhsofapplyingK}
\begin{split}
L_K(i_Y\mu) &= i_{(L_K Y)}\mu + i_Y L_K\mu\\
&= i_{(L_K Y)}\mu.  
\end{split}
\end{equation}
Concerning the right-hand side, we have
\begin{equation*}
L_K(\alpha(X)) = (L_K\alpha)(X) + \alpha(L_K X) = 0 + 0 = 0
\end{equation*}
so that
\begin{equation*}
L_K\left(\frac{1}{\alpha(X)}\right) = 0.
\end{equation*}
Moreover,
\begin{equation*}
L_K(\alpha \wedge dp) = (L_K \alpha) \wedge dp + \alpha \wedge (L_K dp). 
\end{equation*}
Hence, the right-hand side of Equation \eqref{eq:applyingK} is zero. Thus, by Equation \eqref{eq:lhsofapplyingK}, the conclusion follows.
\end{proof}
The results introduced in this section give a general framework for some of the results we will prove throughout the next sections. For the interested reader, Example \ref{eg:MHD-with-Killing-on-M} in Section \ref{sec:extensions} exhibits an explicit guided flow and computation of its companion vector field.

\section{Perturbation of adapted metrics}\label{sec:perturbing-adapted-metrics}

The primary aim of this section is to prove Theorem \ref{thm:introalphafixed1} and \ref{thm:introalphafixed2}. This is accomplished by first developing a methodology for perturbing adapted metrics of $(X,\alpha,\mu)$ to new metrics while maintaining the adapted property. This is done in Section \ref{sec:construction-of-pertubation}. Then, the utility of these perturbations is demonstrated by proving a generalisation of Theorem \ref{thm:introalphafixed1} and \ref{thm:introalphafixed2} which is given by Proposition \ref{prop:for-intro-alpha-fixed-1-and-2}. Finally, in Section \ref{sec:proof-of-main-results}, we prove the main results by demonstrating how they are implied by Proposition \ref{prop:for-intro-alpha-fixed-1-and-2}. Before we begin, recall that on an oriented manifold $M$ with boundary, if $X$ is a vector field, $\alpha$ is a 1-form, and $\mu$ is a positively oriented volume-form on $M$, then a metric $g$ on $M$ is said to be \emph{adapted to $(X,\alpha,\mu)$}, provided that
\begin{equation*}
i_Xg = \alpha, \qquad \mu_g = \mu,
\end{equation*}
where $\mu_g$ denotes the Riemannian volume-form on $(M,g)$.

\subsection{Construction and application of a perturbed metric}\label{sec:construction-of-pertubation}

Let $M$ be an oriented 3-manifold with (possibly empty) boundary. Fix a vector field $X$, a function $p$ and a metric $g$ such that on $(M,g)$, we have
\begin{equation}\label{eq:nearly-Euler}
\Div X = 0, \qquad X \cdot \nabla p = 0, \qquad \curl X \cdot \nabla p = 0.
\end{equation}
Set $\alpha = i_X g$ and $\mu = \mu_{g}$. Consider an open subset 
\begin{equation*}
U \subset \text{Reg}_M(p) \cap \{x \in M : X|_x \neq 0\}
\end{equation*}
so that in particular $(X,\alpha,\mu,p)$ is a guided flow when restricted to $U$. Note that in order for $U$ to be non-empty, we need at least that $dp \neq 0$ and $X \neq 0$. We will consider the local vector field $Y : U \to TM$ defined by
\begin{equation}\label{eq:loc-companion-field}
i_Y \mu|_U = \frac{1}{\alpha(X)|_U} \alpha \wedge dp|_U.    
\end{equation}

We note for the reader that $Y$ may also be written in terms of $g$ as
\begin{equation*}
Y = \frac{1}{g(X,X)|_U} X \times \nabla p |_U
\end{equation*}
where here we recall that $\nabla p$ denotes the gradient of $p$ with respect to the metric $g$ and the cross-product, $\times$, is the one induced by the metric $g$. Equation \eqref{eq:loc-companion-field} is the preferable means of expressing $Y$ because it explicitly only depends on $(X,\alpha,\mu,p)$, making it more compatible with the notion of adapted metric.

Let $\rho : M \to \mathbb{R}$ be a positive function such that $\rho - 1$ is compactly supported in $U$. Consider the symmetric bilinear form $g^{\rho}$ on $M$ such that
\begin{equation}\label{eq:metric-awayU}
g^{\rho} = g \text{ on } M \backslash U 
\end{equation}
and, on $U$,
\begin{equation}\label{eq:metric-coefficients}
\begin{aligned}
&g^{\rho}(X,X) &=& &\alpha(X), & &g^{\rho}(Y,\nabla p) &=& &0,\\
&g^{\rho}(X,Y) &=& &0, & &g^{\rho}(Y,Y) &=& &\rho g(Y,Y),\\
&g^{\rho}(X,\nabla p) &=& &0, & &g^{\rho}(\nabla p,\nabla p) &=& &\frac{1}{\rho} g(\nabla p,\nabla p).
\end{aligned}    
\end{equation}

Now because $i_Xg = \alpha$, we have
\begin{equation*}
\alpha(\nabla p) = g(X,\nabla p) = dp(X) = 0
\end{equation*}
and so the first column of Equation \eqref{eq:metric-coefficients} implies that 
\begin{equation}\label{eq:perturbed-is-adapted}
i_Xg^{\rho} = \alpha.
\end{equation}
Moreover, observe that the frame
\begin{equation*}
(X|_U,Y|_U,\nabla p|_U)
\end{equation*}
is orthogonal for $g^{\rho}$ and that, on $U$,
\begin{equation*}
g^{\rho}(X,X) > 0, \qquad g^{\rho}(Y,Y) > 0, \qquad g^{\rho}(\nabla p,\nabla p) > 0.
\end{equation*}
Thus, $g^{\rho}$ is a metric on $M$. Also note that the frame $(X|_U,Y|_U,\nabla p|_U)$ is $g$-orthogonal. 

We now compute the induced volume element $\mu^{\rho} = \mu_{g^{\rho}}$ from the metric $g^{\rho}$ on $M$. First, we observe that
\begin{align*}
i_Y i_X \mu|_U = dp|_U.        
\end{align*}
Thus, on $U$,
\begin{equation*}
\mu(X,Y,\nabla p) = dp(\nabla p) = g(\nabla p,\nabla p) > 0.
\end{equation*}
Hence, $(X|_U,Y|_U,\nabla p|_U)$ is positively oriented. We therefore have, on $U$,
\begin{equation*}
\mu(X,Y,\nabla p) = \mu_{g}(X,Y,\nabla p) = \|X\|~\|Y\|~\|\nabla p\| = \sqrt{\alpha(X)}~\|Y\|~\|\nabla p\|
\end{equation*}
where $\|.\|$ denotes the fibre-wise norm induced by $g$ on $TM$. On the other hand, letting $\|.\|^{\rho}$ denote the fibre-wise norm induced by $g^{\rho}$ on $TM$, we have
\begin{align*}
\mu_{\rho}(X,Y,\nabla p) &= \|X\|^{\rho}~\|Y\|^{\rho}~\|\nabla p\|^{\rho}\\
&= \sqrt{\alpha(X)}~\sqrt{\rho}\|Y\|~\sqrt{\rho^{-1}}\|\nabla p\|\\
&= \sqrt{\alpha(X)}~\|Y\|~\|\nabla p\|.
\end{align*}
It follows that
\begin{equation}\label{eq:modified-volume}
\mu^{\rho} = \mu,
\end{equation}
and thus for a positive smooth $\rho : M \to \RR$ with $\rho-1$ compactly supported in $U$, the metric $g^{\rho}$ is adapted to $(X,\alpha,\mu)$ by equations \eqref{eq:perturbed-is-adapted} and \eqref{eq:modified-volume}. We will use these perturbed adapted metrics for varying $\rho$ to establish the following proposition, which will be the key step in the proof of Theorem \ref{thm:introalphafixed1} and Theorem \ref{thm:introalphafixed2}.

\begin{proposition}\label{prop:for-intro-alpha-fixed-1-and-2}
Let $M$ be an oriented 3-manifold with (possibly empty) boundary. Let $g_0$ be a metric on $M$ and $\mu = \mu_{g_0}$. Let $X$ be a vector field on $M$ satisfying Equation \eqref{eq:nearly-Euler} on $(M,g_0)$ for some $p \in C^{\infty}(M)$. Suppose that $S$ is an embedded surface in $M$ such that $p$ is constant and regular on $S$ with $X$ tangent to $\partial S$. Set $\alpha = i_X g_0$. Then, the following hold.
\begin{enumerate}
    \item Assume $X|_S \neq 0$ and let $V$ be a neighbourhood of $S$ such that the level sets of $p|_V$ are compact. Then, there is a $C^0$-open and $C^{\infty}$-dense subset of $\cU \subset \cM_{(X,\alpha,\mu)}(M)$ such that for each $g \in \cU$, there are no non-trivial Killing fields $K$ on $(M,g)$ with $L_K p|_V = 0$ and with $K$ being tangent to $\partial S$.
    \item Assume $X|_S$ is nowhere-vanishing. Then, there is a $C^2$-open and $C^{\infty}$-dense subset of $\cU \subset \cM_{(X,\alpha,\mu)}(M)$ such that for each $g \in \cU$, there are no non-trivial Killing fields $K$ on $(M,g)$ with $K$ tangent to $S$ and $\partial S$.
\end{enumerate}
\end{proposition}

In Section \ref{sec:proof-of-main-results}, we will see that part (i) of Proposition \ref{prop:for-intro-alpha-fixed-1-and-2} implies Theorem \ref{thm:introalphafixed1} and part (ii) of the proposition implies Theorem \ref{thm:introalphafixed2}. The remainder of this section is dedicated to proving Proposition \ref{prop:for-intro-alpha-fixed-1-and-2}. 

Both parts (i) and (ii) of Proposition \ref{prop:for-intro-alpha-fixed-1-and-2} rely on a common lemma. We define the following basic notion.

\begin{definition}
Let $S$ be a surface and let $x \in S$. Then, a neighbourhood $W$ of $x$ is said to be a \emph{proper coordinate neighbourhood} of $x$ if $W \subsetneq S$ and there exists a chart $(W,\varphi)$ such that $\varphi(x) = 0$ and for some $\delta > 0$ it holds that:
\begin{enumerate}
    \item if $x \in \partial S$, $\varphi(W) = B_{\delta}(0) \cap \HH^2$ 
    \item and otherwise, $\varphi(W) = B_{\delta}(0)$
\end{enumerate}
where $B_{\delta}(0)$ is a ball centered at zero of radius $\delta$ and $\HH$ is the upper half-plane in $\mathbb{R}^2$.
\end{definition}

\begin{lemma}\label{lem:generic-non-symmetry}
There exists a $C^\infty$-dense and $C^0$-open set $\cV \subset C^\infty(S)$ such that, for any $f\in \cV$ and any non-vanishing vector field $X$ on $S$ tangent to $\partial S$, $L_X f \neq 0$.
\end{lemma}

\begin{proof}
Let $\cV$ denote the set of smooth functions $f \in C^{\infty}(S)$ for which there exists a point $x \in S$ and a proper coordinate neighbourhood $W$ of $x$ such that $f(x) > f(x')$ for all $x' \in S\backslash W$. First, we establish that $\cV$ is $C^0$-open in $C^\infty(S)$. Let $\text{CoN}(S)$ denote the set of proper coordinate neighbourhoods $S$. For any such coordinate neighbourhood $W \in \text{CoN}(S)$, we have that $S\backslash W$ is a compact (and non-empty) subset of $S$ so we may consider the function $N_W : C^{\infty}(S) \to \RR$ given by
\begin{equation*}
N_W(f) = \max_{x \in S} f(x) - \max_{x' \in S \backslash W}f(x').
\end{equation*}
Note that $N_W$ is continuous with respect to the $C^0$-topology on $C^{\infty}(S)$ (and the usual topology on $\RR$). We may write
\begin{equation*}
\cV = \cup_{W \in \text{CoN}(S)}\{f \in C^{\infty}(S) : N_W f > 0\}
\end{equation*}
which is open as the union of open sets. 

Now we prove $C^{\infty}$-density. Let $f \in C^{\infty}(S)$ and let $x \in S$ be a point that maximises $f$. Then, let $W$ be a coordinate neighbourhood of $x$ and $\rho \in C^{\infty}(S)$ be compactly supported in $W$ such that $\rho(x) > 0$. Then, clearly $f + \rho \in \cV$. In particular, since $\rho$ may be taken $C^\infty$-close to $0$, this proves $C^{\infty}$-density.

Lastly, we prove the statement about vector fields. Let $f \in \cV$ and suppose for the sake of contradiction that there exists a nowhere-vanishing vector field $X$ on $S$ tangent to $\partial S$ such that $L_X f = 0$. Because $X$ is tangent to $\partial S$ and nowhere-vanishing there, we have that $f$ is constant on $\partial S$. Writing
\begin{equation*}
y = \max_{x \in S} f(x), \qquad y' = \max_{x' \in S \backslash V}f(x'),    
\end{equation*}
we have $y' < y$ and $[y',y] \subset f(S)$ by continuity. Hence, $(y',y]\subset f(V)$. By Sard's theorem, there exists a regular value $y^*$ of $f$ such that $y' < y^* < y$ where $y^*$ is not a value of $f$ attained on $\partial S$. Now, on the one hand, $f^{-1}(y^*)$ is a compact subset of $S$, on the other hand, $f^{-1}(y^*) \subset \Int W$ where $\Int W$ denotes the interior of $W$ as an open submanifold of $S$. Thus, $f^{-1}(y^*)$ is a compact embedded 1-dimensional submanifold of $\Int W$. The connected components of $f^{-1}(y^*)$ are all circles, which bound smooth disks in $\Int W$ since $W$ is a proper coordinate neighbourhood. Let $D$ be one of these disks, and assume by contradiction that $L_Xf=0$. Since the boundary of $D$ is a regular connected component of a level set of $f$, $X$ is tangent to it. The Poincar\'e-Hopf Theorem implies that $X$ admits a zero in $D$, reaching a contradiction with the fact that $X$ is non-vanishing.
\end{proof}
\begin{remark}
    One can choose a larger set than $\cV$ in the previous lemma, namely those functions $f \in C^\infty(S)$ for which there exist two open disks $W_1,W_2$ with $\overline{W_1} \subsetneq W_2$ such that there is a point $x\in W_1$ for which $f(x)>f(x')$ for every $x'\in W_2\setminus W_1$.
\end{remark}
We now prove Proposition \ref{prop:for-intro-alpha-fixed-1-and-2}.

\begin{proof}[Proof of \ref{prop:for-intro-alpha-fixed-1-and-2}]
The proof of parts (i) and (ii) are similar. Both proofs rely on the subset $\cV$ of $C^{\infty}(S)$ considered in Lemma \ref{lem:generic-non-symmetry}. Another similarity is that the assumptions in both parts include $X|_S \neq 0$. By Lemma \ref{lemma:XisP-harmonic}, we therefore have in the setting of either (i) or (ii) that $H_{\text{dR}}^1(S) \neq \{0\}$. In particular, if one has a metric $g$ on $M$ and $K$ is a Killing field on $M$ tangent to $S$ and $\partial S$, the vector field $K_S$ which is related to $K$ by the inclusion $i : S \subset M$, is a Killing field on $(S,i^*g)$ tangent to $\partial S$ by Lemma \ref{lem:vanishingonsurface}. Therefore, if $K_S|_x = 0$ for some $x \in S$, then $K_S = 0$ by Lemma \ref{lem:vanishingonsurface} and thus, $K$ vanishes in all $M$.\\

\textbf{Proof of part (i).} Firstly, we note that it immediately follows from Proposition \ref{prop:KcommX} that if $X|_x = 0$ for some $x \in S$, then part (i) is immediately true with $\cU = \cM_{(X,\alpha,\mu)}(M)$.

We henceforth assume that $X|_S$ is nowhere-vanishing so that $S \subset U$ where $U = \{x \in M : X|_x \neq 0\} \cap V$. Consider the local companion vector field $Y : U \to TM$ defined as in Equation \eqref{eq:loc-companion-field}. Define the function $N : \cM_{(X,\alpha,\mu)}(M) \to C^{\infty}(S)$ by
\begin{equation*}
N(g) = g|_S(Y|_S,Y|_S).    
\end{equation*}
Recall the local vector field $\tilde{X} : U \to TM$ given by
\begin{equation*}
\tilde{X} = \frac{1}{\alpha(X)|_U} X|_U.
\end{equation*}
Both $\tilde{X}$ and $Y$ are tangent to $S$ and induce vector fields $\tilde{X}_S$ and $Y_S$. We consider the set $\cV$ as in the above Lemma \ref{lem:generic-non-symmetry} and define
\begin{equation*}
\cU = N^{-1}(\cV).
\end{equation*}
The proof will be complete once the following properties of $\cU$ can be shown.
\begin{enumerate}
    \item[1)] For each $g \in \cU$, there are no non-trivial Killing fields $K$ on $(M,g)$ such that $L_K p = 0$ with $K$ tangent to $\partial S$,
    \item[2)] $\cU$ is $C^0$-open in $\cM_{(X,\alpha,\mu)}(M)$, and
    \item[3)] $\cU$ is $C^{\infty}$-dense in $\cM_{(X,\alpha,\mu)}(M)$.
\end{enumerate}
For the first property, if $g \in \cU$, then let $K$ be a Killing field on $(M,g)$ such that $L_K p|_V = 0$ with $K$ tangent to $\partial S$. Summarising with the aid of Remark \ref{rem:killing-on-adapted-metric} and Proposition \ref{prop:KcommX}, we have the invariance 
\begin{equation*}
L_K \alpha|_V = 0, \qquad L_K \mu = 0, \qquad L_K p|_V = 0.
\end{equation*}
Thus, by Lemma \ref{lem:companionispreserved},
\begin{equation*}
L_K Y|_V = 0.    
\end{equation*}
On the other hand, $K$ is a Killing field on $M$ and thus, restricting to $S$, we have that
\begin{equation*}
L_{K_S} N(g) = 0
\end{equation*}
where $K_S$ is the vector field related to $K$ via the inclusion $i : S \subset M$. However, because $N(g) \in \cV$, Lemma \ref{lem:generic-non-symmetry} implies that $K_S$ must have a zero. However, by Lemma \ref{lem:vanishinginsurface}, because $K_S$ is Killing on $(S,i^*g)$, this implies that $K_S = 0$ and therefore $K|_S = 0$. But then, from Lemma \ref{lem:vanishingonsurface}, it follows that $K = 0$ on $M$. 

To deduce the second property, observe that the functional $N : \cM_{(X,\alpha,\mu)}(M) \to C^{\infty}(S)$ is continuous with respect to the $C^0$-topologies on $\cM_{(X,\alpha,\mu)}(M)$ and $C^{\infty}(S)$ and therefore, by Lemma \ref{lem:generic-non-symmetry}, $\cU = N^{-1}(\cV)$ is $C^0$-open in $\cM_{(X,\alpha,\mu)}(M)$. 

Finally, for the third property, fix a metric $g \in \cM_{(X,\alpha,\mu)}(M)$. Then, $N(g) \in C^{\infty}(S)$ so by Lemma \ref{lem:generic-non-symmetry} there exists a positive function $\rho_S \in C^{\infty}(S)$ that is $C^{\infty}$-close to $1$ (in $C^{\infty}(S)$) such that $\rho_S \cdot N(g) \in \cV$. We may extend $\rho_S$ to a positive function $\rho \in C^{\infty}(M)$ such that $\rho-1$ is compactly supported in $U$ whereby $\rho$ is $C^{\infty}$-close to $1$ (in $C^\infty(M)$). We then consider the metric $g^{\rho} \in \cM_{(X,\alpha,\mu)}(M)$ defined as in Equations \eqref{eq:metric-awayU} and \eqref{eq:metric-coefficients}, which is $C^{\infty}$-close to $g$ and satisfies
\begin{equation*}
N(g^{\rho}) = \rho_S\cdot N(g) \in \cV  
\end{equation*}
so that $g^{\rho} \in \cU$.

We now move to part (ii) of the proposition, whose proof has the same structure as that of part (i).\\

\textbf{Proof of part (ii).}
We consider the function $s : \cM_{(X,\alpha,\mu)}(M) \to C^{\infty}(S)$ where for each $g \in \cM_{(X,\alpha,\mu)}(M)$, $s(g)$ is the scalar curvature of the metric $i^*g$ on $S$, where $i : S \subset M$ is the inclusion. Let
\begin{equation*}
\cU = s^{-1}(\cV).    
\end{equation*}
Our proof will be complete once we prove the following properties of $\cU$.
\begin{enumerate} 
    \item[1)] For each $g \in \cU$, there are no non-trivial Killing fields $K$ on $(M,g)$ such that $K$ is tangent to $S$ and its boundary,
    \item[2)] $\cU$ is $C^2$-open in $\cM_{(X,\alpha,\mu)}(M)$, and
    \item[3)] $\cU$ is $C^{\infty}$-dense in $\cM_{(X,\alpha,\mu)}(M)$.
\end{enumerate}

For property 1), let $g \in \cU$ and let $K$ be a Killing field on $(M,g)$ such that $K$ tangent $S$ to $\partial S$. Considering the vector field $K_S$ related to $K$ by the inclusion $i : S \subset M$, writing $\kappa = s(g)$, we have the invariance
\begin{equation*}
d\kappa(K_S) = 0.
\end{equation*}
However, $\kappa = s(g) \in \cV$. Thus, by Lemma \ref{lem:generic-non-symmetry}, it must be that $K_S$ has a zero on $S$. But, as noted earlier, $H_{\text{dR}}^{1}(S) \neq \{0\}$ and so $K_S = 0$ by Lemma \ref{lem:vanishinginsurface}. Hence, $K|_S = 0$. Lemma \ref{lem:vanishingonsurface} then implies that $K = 0$.

For the second property, note that $\cV$ is $C^0$-open in $C^{\infty}(S)$ by Lemma \ref{lem:generic-non-symmetry}. On the other hand, the map $s : \cM_{(X,\alpha,\mu)}(M) \to C^{\infty}(S)$ is continuous with respect to the $C^2$ topologies on $\cM_{(X,\alpha,\mu)}(M)$ and $C^0$ topology on $C^{\infty}(S)$, since the scalar curvature depends continuously on the components of $g$, its Christoffel symbols and their derivatives. Hence, $\cU$ is $C^2$-open in $\cM_{(X,\alpha,\mu)}(M)$.

Fix some $g \in \cM_{(X,\alpha,\mu)}(M)$, and let us prove property 3). Choose a point $x \in S$ which maximises $s(g)$ and let $\rho \in C^{\infty}(M)$ be a smooth and positive function such that $\rho-1$ is compactly supported in $M$, whereby $(\rho-1)|_S$ is compactly supported in a proper coordinate neighbourhood $W$ of $x$. By Equation \eqref{eq:metric-coefficients}, we have that the difference
\begin{equation*}
s(g^\rho)-s(g)    
\end{equation*}
is zero outside $W$. We will be done by Lemma \ref{lem:generic-non-symmetry} if there exists a function $\rho \in C^\infty(M)$ that is arbitrarily $C^\infty$ close to $1$ so that the difference
\begin{equation*}
s(g^\rho)-s(g)    
\end{equation*}
is positive at $x$, then for such $\rho$, $s(g^\rho) \in \cV$ so that $g^\rho \in \cU$. We now show this can be done. Henceforth, and as in the proof of Proposition \ref{prop:for-intro-alpha-fixed-1-and-2} (i), we will consider the vector fields $\tilde{X},Y : U \to TM$ and in particular the consequent vector fields $\tilde{X}_S$ and $Y_S$. Setting $E = g(X,X)|_S$, $G = g|_S(Y|_S,Y|_S)$, $g_S = i^*g$, and $g_S^\rho = i^*g^{\rho}$, by Proposition \ref{prop:companionflow}, we have
\begin{align*}
g_S &= E \omega^2 + G \eta^2\\
g^{\rho}_S &= E \omega^2 + \rho|_S G \eta^2
\end{align*}
where $(\omega,\eta)$ are the dual co-frame to $(\tilde{X}_S,Y_S)$. We wish to compute the scalar curvatures near $x$ of $g_S$ and $g^{\rho}_S$. The following claim will aid in this computation.

\begin{claim} If $x \in \partial S$ (or $x \in \Int S$), then there exists $\epsilon > 0$ and a diffeomorphism onto its open image 
\begin{equation*}
H : (-\epsilon,\epsilon) \times I \to S
\end{equation*}
where $I = [0,\epsilon)$ or $I = (-\epsilon,0]$ (respectively $I = (-\epsilon,\epsilon)$) where for each $s \in (-\epsilon,\epsilon)$, $t \mapsto H(s,t)$ is an integral curve of $Y_S$ initialised at $\Psi_s(x)$, where $\Psi$ denotes the global flow of $\tilde{X}_S$.
\end{claim}
\begin{proof}[Proof of claim] Recall that $\tilde{X}_S$ is tangent to $\partial S$ and so by compactness of $S$, the vector field $\tilde{X}_S$ has a globally defined flow. Now, suppose $x \in \partial S$. We have that $\partial S$ is an embedded submanifold of $S$ and because the image of the integral curve $s \mapsto \Psi_s(x)$ is contained in $\partial S$, the inverse function theorem (in $\partial S$) implies that for any sufficiently small $\epsilon > 0$, the resulting curve $\gamma : (-\epsilon,\epsilon)\to \partial S$ is a diffeomorphism onto its open image in $\partial S$. If $x \in \Int S$, then the image of the integral curve $s \mapsto \Psi_s(x)$ is contained in $\Int S$ and it is standard that for any sufficiently small $\epsilon > 0$, the resulting curve $\gamma : (-\epsilon,\epsilon)\to \Int S$ is an embedding. Then, applying the Boundary Flowout Theorem when $x \in \partial S$ (since $Y_S$ is transverse to $\partial S$) and the standard Flowout Theorem when $x \in \Int S$, we obtain our desired map $H$.
\end{proof}

With the claim proven, note in addition that because $\tilde{X}_S$ and $Y_S$ commute, $\omega$ and $\eta$ are closed 1-forms. It is easy to check that in fact $H^*\omega = du$ and $H^*\eta = dv$ where $u,v : (-\epsilon,\epsilon) \times I \to \RR$ are the standard Cartesian coordinates. Hence, we can write
\begin{align*}
H^*g_S &= \tilde{E} du^2 + \tilde{G} dv^2\\
H^*g^{\rho}_S &= \tilde{E} du^2 + \tilde{\rho} \tilde{G}dv^2
\end{align*}
where $\tilde{E} = E \circ H$, $\tilde{G} = G \circ H$, and $\tilde{\rho} = (\rho|_S) \circ H$. By isometric invariance of scalar curvature, we have the equalities
\begin{align*}
&s(g)(x)& &=& &\cS(g_S)(x)& &=& &\cS(H^*g_S)(0)&\\
&s(g^\rho)(x)& &=& &\cS(g_S)(x)& &=& &\cS(H^*g^{\rho}_S)(0)&
\end{align*}
where $\cS(g')$ denotes the scalar curvature of a metric $g'$ on a manifold with boundary. By the Gauss formula for scalar curvature, we obtain
\begin{equation*}
\cS(H^*g
) = -\frac{1}{\tilde{E}\tilde{G}} \left(\tilde{E}_{vv} + \tilde{G}_{uu}-\frac{1}{2}\left(\frac{\tilde{E}_u \tilde{G}_u + \tilde{E}_v^2}{\tilde{E}}+\frac{\tilde{E}_v \tilde{G}_v+\tilde{G}^2_u}{\tilde{G}}\right) \right)
\end{equation*}
on the other hand,
\begin{align*}
\cS(H^*g^{\rho}_S) &= -\frac{1}{\tilde{E}(\tilde{\rho}\tilde{G})} \left(\tilde{E}_{vv} + (\tilde{\rho}\tilde{G})_{uu}-\frac{1}{2}\left(\frac{\tilde{E}_u (\tilde{\rho}\tilde{G})_u + \tilde{E}_v^2}{\tilde{E}}+\frac{\tilde{E}_v (\tilde{\rho}\tilde{G})_v+(\tilde{\rho}\tilde{G})^2_u}{\tilde{\rho}\tilde{G}}\right) \right)\\
&= -\frac{1}{\tilde{E}\tilde{G}} \left(\frac{\tilde{E}_{vv} + (\tilde{\rho}\tilde{G})_{uu}}{\tilde{\rho}}-\frac{1}{2}\left(\frac{\tilde{E}_u (\tilde{\rho}\tilde{G})_u + \tilde{E}_v^2}{\tilde{\rho}\tilde{E}}+\frac{\tilde{E}_v (\tilde{\rho}\tilde{G})_v+(\tilde{\rho}\tilde{G})^2_u}{\tilde{\rho}^2\tilde{G}}\right) \right).
\end{align*}
Now, we may choose $\rho$ so that $\tilde{\rho} = 1 + \rho_0 u^2$ where $\rho_0$ is a bump function which equals a (small) constant $c$ and compactly supported in a neighbourhood $W_0 \subset H^{-1}(W)$ of $0$. Then, $\tilde{\rho}$ satisfies
\begin{equation*}
\tilde{\rho}(0) = 1, \qquad \tilde{\rho}_u(0) = 0 = \tilde{\rho}_v(0), \qquad \tilde{\rho}_{uu}(0) = 2c.
\end{equation*}
Hence, at $0$, we have
\begin{align*}
\cS(H^*g^{\rho}_S)(0) &= -\frac{1}{\tilde{E}\tilde{G}} \left(\tilde{E}_{vv} + \tilde{G}_{uu} + 2c\tilde{G}_{uu} -\frac{1}{2}\left(\frac{\tilde{E}_u \tilde{G}_u + \tilde{E}_v^2}{\tilde{E}}+\frac{\tilde{E}_v \tilde{G}_v +\tilde{G}^2_u}{\tilde{G}}\right) \right)\bigg|_0.
\end{align*}
In particular,
\begin{equation*}
\cS(H^*g^{\rho}_S)(0) - \cS(H^*g_S)(0) = \frac{-2c}{\tilde{E}\tilde{G}|_0}.
\end{equation*}
and thus
\begin{equation*}
s(g^{\rho})(x) - s(g)(x) = \frac{-2c}{\tilde{E}\tilde{G}|_0} > 0
\end{equation*}
if $c < 0$. This proves (iii) and in turn, proves Proposition \ref{prop:for-intro-alpha-fixed-1-and-2} (ii). 
\end{proof}

We now are in a position to prove Theorems \ref{thm:introalphafixed1} and \ref{thm:introalphafixed2}.

\subsection{Proof of Theorems \ref{thm:introalphafixed1} and \ref{thm:introalphafixed2}}\label{sec:proof-of-main-results}

Theorem \ref{thm:introalphafixed1} is an application of Proposition \ref{prop:for-intro-alpha-fixed-1-and-2} (i).

\begin{proof}[Proof of Theorem \ref{thm:introalphafixed1}]
We first recall the setting. Assume $M$ is a compact connected oriented 3-manifold with boundary. Fix a metric $g_0$ on $M$ and set $\mu = \mu_{g_0}$. Let $(X,p)$ be an MHD equilibrium on $(M,g_0)$ with pressure function $p$ which is constant on $\partial M$ but with $\nabla p \neq 0$ on $M$. By Sard's Theorem, because $p$ is constant on $\partial M$ and $dp \neq 0$, there exists a closed regular level set of $p$ in $\Int M$. Let $S$ be a connected component of this regular level set.

Now let $g \in \cM_{(X,\alpha,\mu)}(M)$ be an adapted metric. Let $K$ be a Killing field on $(M,g)$. Suppose that $L_K X = 0$. Then, because $p$ is constant on $\partial M$ and $M$ is compact, $p$ has a critical point in $M$. Thus, by Lemma \ref{lem:tangenttobernoulli}, $L_K p = 0$. Moreover, because $\partial S = \emptyset$, $K$ is tangent to $\partial S$. Therefore if we show that there are no non-trivial Killing fields $K$ on $(M,g)$ such that $L_K p = 0$ with $K$ tangent to $\partial S$, then there are no non-trivial Killing fields $K$ on $(M,g)$ such that $L_K X = 0$ neither. Part (i) of Proposition \ref{prop:for-intro-alpha-fixed-1-and-2} completes the proof of the theorem.
\end{proof}

Let us proceed with the proof of Theorem \ref{thm:introalphafixed2}, as an application of part (ii) of Proposition \ref{prop:for-intro-alpha-fixed-1-and-2}.

\begin{proof}[Proof of Theorem \ref{thm:introalphafixed2}]
We first recall the setting. Assume $M$ is a compact connected oriented 3-manifold with boundary. Fix a metric $g_0$ on $M$ and set $\mu = \mu_{g_0}$. Let $(X,p)$ be an MHD equilibrium on $(M,g_0)$ with pressure function $p$ which is constant and regular on the connected components of $\partial M$. We set $\alpha = i_X g_0$.

Now, setting $S$ to be a connected component of $\partial M$, because $\partial S = \emptyset$, we have for $g \in \cM_{(X,\alpha,\mu)}(M)$ that a Killing field $K$ on $(M,g)$ is tangent to $S$ if and only if it is tangent to $S$ and its boundary. Thus, with this choice of $S$, Proposition \ref{prop:for-intro-alpha-fixed-1-and-2} (ii) immediately implies that there is a $C^2$-open and $C^{\infty}$-dense subset of $\cU \subset \cM_{(X,\alpha,\mu)}(M)$ such that each $g \in U$, there are no non-trivial Killing fields $K$ on $(M,g)$ such that $K$ is tangent to $\partial M$.
\end{proof}

\section{Extensions and examples of the main results}\label{sec:extensions}

In this section, we provide some extensions and examples of the main results which take place on a compact connected oriented 3-manifold $M$ with boundary. It will be clear to the reader that other variations exist but for the sake of brevity, we only list a few of them. 

It is natural to ask whether a variant of Theorem \ref{thm:introalphafixed1} holds if one does not require that the pressure function is constant on the boundary. This assumption, however, enabled us to conclude that the Killing fields in the proof of Theorem \ref{thm:introalphafixed1} were also tangent to the pressure function of the steady Euler flow. In the following corollary of Proposition \ref{prop:for-intro-alpha-fixed-1-and-2}, we assume this to be the case, along with some boundary conditions.

\begin{corollary}
Fix a metric $g_0$ on $M$ and set $\mu = \mu_{g_0}$. Let $(X,p)$ be an MHD equilibrium on $(M,g_0)$ which is tangent to $\partial M$. Assume the pressure function is such that $dp \neq 0$. Then, setting $\alpha = i_Xg_0$, there is a $C^0$-open and $C^{\infty}$-dense subset of $\cU \subset \cM_{(X,\alpha,\mu)}(M)$ such that for any $g \in \cU$, there does not exist non-trivial Killing fields $K$ on $(M,g)$ tangent to $\partial M$ such that $L_K p = 0$.
\end{corollary}

\begin{proof}
Because $dp \neq 0$, there exists a $\partial$-regular value of $p$. We therefore may fix a connected component $S$ of a $\partial$-regular level set. In particular, for any vector field $Z$ tangent to $\partial M$ such that $L_Z p = 0$, we have that $Z$ is tangent to $S$ and its boundary. We immediately obtain from Proposition \ref{prop:for-intro-alpha-fixed-1-and-2} (i) the result.
\end{proof}

We will now consider some examples of ideal MHD equilibria where the pressure function $p$ satisfies $dp = 0$. First consider Beltrami fields with a smooth proportionality factor on $(M,g_0)$. Recall, these are vector fields $X$ such that
\begin{equation*}
\Div X = 0, \qquad \curl X = \lambda X 
\end{equation*}
on $(M,g_0)$ for some $\lambda \in C^{\infty}(M)$. Also note, writing $\alpha = i_X g_0$ and $\mu = \mu_{g_0}$, it follows immediately from Equation \eqref{eq:curlandcrossprod} that if one has a metric $g \in \cM_{(X,\alpha,\mu)}(M)$, then $X$ remains a Beltrami field with proportionality factor $\lambda$ on $(M,g)$.

\begin{corollary}\label{cor:beltrami-fields}
Fix a metric $g_0$ on $M$ and set $\mu = \mu_{g_0}$. Let $X$ be a not identically zero Beltrami field on $(M,g_0)$ tangent to $\partial M$ with smooth proportionality factor $\lambda$ and suppose that $d\lambda \neq 0$. Setting $\alpha = i_Xg_0$, there is a $C^0$-open and $C^{\infty}$-dense subset of $\cM_{(X,\alpha,\mu)}(M)$ which do not support non-trivial Killing fields $K$ on $(M,g)$ tangent to $\partial M$ and satisfying $L_K X = 0$.
\end{corollary}

\begin{proof}
Because $d\lambda \neq 0$, there exists a $\partial$-regular value $y$ of $\lambda$ such that $\lambda^{-1}(y)$ intersects $U = \{x \in M : X|_x \neq 0\}$. Indeed, this is because $U$ is a dense subset of $M$ (in fact, $\Int M \setminus U$ has a Hausdorff dimension of at most 1, see \cite[Proposition 2.1]{Ger}). We therefore may fix a connected component $S$ of $\lambda^{-1}(y)$ intersecting $U$. Then, fixing $g \in \cM_{(X,\alpha,\mu)}(M)$, if $K$ is a Killing field on $(M,g)$ tangent to $\partial M$, then as per Remark \ref{rmk:Beltrami}, $L_K \lambda = 0$, and therefore $K$ is tangent to $S$ and $\partial S$. From this, the result follows immediately from Proposition \ref{prop:for-intro-alpha-fixed-1-and-2} (i).
\end{proof}

For our last extension of the main results, we talk about Neumann harmonic vector fields on $M$, also known as vacuum fields. If $g$ is a metric on $M$, a vector field $X$ on $M$ is said to be a Neumann harmonic vector field on $(M,g)$ provided
\begin{equation*}
\Div X = 0 \text{ on } M, \qquad \curl X = 0 \text{ on } M, \qquad X \cdot n = 0 \text{ on } \partial M, 
\end{equation*}
where $n$ denotes the outward unit normal. If $g'$ is a metric adapted to $(X,\alpha = i_X g,\mu)$, then clearly $X$ is a Neuman harmonic field on $(M,g')$ also. The analog of Theorem \ref{thm:introalphafixed2} (from Proposition \ref{prop:for-intro-alpha-fixed-1-and-2} (ii)) is immediately obvious in this case.

\begin{corollary}\label{cor:vacuumC^2}
Fix a metric $g_0$ on $M$ and set $\mu = \mu_{g_0}$. Let $X$ be a Neumann vacuum vector field on $(M,g_0)$ tangent to $\partial M$. Suppose that $\partial M \neq \emptyset$ and $X$ has a first integral $f$ which is constant and regular on the connected components of $\partial M$. Then, setting $\alpha = i_Xg_0$, there is a $C^2$-open and $C^{\infty}$-dense subset of $\cM_{(X,\alpha,\mu)}(M)$ which do not support non-trivial Killing fields $K$ tangent on $(M,g)$ to $\partial M$ and satisfying $L_K X = 0$.
\end{corollary}

As discussed in Section \ref{sec:Grad conjecture}, Grad was interested not only in counter-examples to Conjecture \ref{conj:Grad}, but also families of such examples. Here we describe how Theorem \ref{thm:introalphafixed1} can be used to generate families of examples.\\

\paragraph{\textbf{Smooth families of solutions.}}
Let $g_0$ be a metric on $M$ and $\mu$ be the associated volume form on $M$. Assume that there exists an MHD equilibrium $(X,p)$ on $(M,g_0)$ with $dp \neq 0$ and with $p$ being constant on $\partial M$. Then, we may fix a 2-torus $S \subset \Int M$ on which $p$ is constant and regular and there exist neighbourhoods $V \subsetneq M$ of $S$ such that $p|_V$ has compact level sets. Let $V$ be such a neighbourhood of $S$ and suppose that there exists a continuous family $\{(X_t,p_t)\}_{t \in \RR}$ of MHD equilibria on $(M,g_0)$ such that $(X_0,p_0) = (X,p)$, and, for $t \in \RR$,
\begin{equation*}
(X_t|_V,p_t|_V) = (X|_V,p|_V)     
\end{equation*}
with $p_t$ being constant on $\partial M$. By continuity of the family, we mean that the map $M \times \RR \ni (x,t) \to (X_t|_x,p_t(x)) \in TM \times \RR$ is continuous. 

\begin{example}\label{ex:families}
An easy example of such a family can be constructed in the flat three-torus $T^3$, i.e. with the metric $g=dx^2+dy^2+dz^2$. Set
\begin{equation*}
X = a(z) \pp{}{x}+b(z) \pp{}{y}, \qquad p=\frac{1}{2}(a^2(z)+b^2(z)).
\end{equation*}
One easily sees that for any $a,b \in C^{\infty}(T)$ the vector field $X$ is an MHD equilibrium with pressure $p$ (we can even modify $a$ and $b$ smoothly while keeping the same $p$).
\end{example}

Continuing with the general construction, set $C = M \backslash V$, consider the $C^0$-open subset
\begin{equation*}
\cM_{(g_0,C)}(M) = \{g \in \cM(M) : g|_C = g_0|_C\}
\end{equation*}
of all metrics $\cM(M)$. Setting $\alpha = i_X g_0$, consider the $C^0$-open subset
\begin{equation*}
\cM_{(X,\alpha,\mu,g_0,C)}(M)  = \cM_{(g_0,C)}(M) \cap \cM_{(X,\alpha,\mu)}(M)
\end{equation*}
of $\cM_{(X,\alpha,\mu)}(M)$. Observe that for $g \in \cM_{(X,\alpha,\mu,g_0,C)}(M)$,
\begin{equation*}
i_{X_t}g = i_{X_t}g_0
\end{equation*}
for all $t \in \RR$. Hence, $\{(X_t,p_t)\}_{t \in \RR}$ is a family of MHD equilibria on $(M,g)$. 

On the other hand, because $\cM_{(X,\alpha,\mu,g_0,C)}(M)$ is a (non-empty) open subset of $\cM_{(X,\alpha,\mu)}(M)$, it follows by Proposition \ref{prop:for-intro-alpha-fixed-1-and-2} (i) that there is a $C^0$-open and $C^{\infty}$-dense subset $\cU$ of $\cM_{(X,\alpha,\mu,g_0,C)}(M)$ such that for each $g \in \cU$, there exist no non-trivial Killing fields $K$ on $(M,g)$ such that $L_K p|_V = 0$.

In particular, for any such $g \in \cU$, for $t \in \RR$, if $K$ was a Killing field on $(M,g)$ with $L_K X_t = 0$ or $L_K p_t = 0$ then because $p_t$ must have a critical point, we have that $L_K X_t = 0 \implies L_K p_t = 0$. Hence, it follows in either case that $L_K p|_V = 0$. Therefore, there are no Killing symmetries on $(M,g)$ for this family for such a metric.

\paragraph{\textbf{Adapted metrics with Killing symmetries.}}

One may question whether Theorem \ref{thm:introalphafixed1} is a reflection of the fact that it is generic for metrics, in the space of all metrics, to admit no non-trivial Killing fields. That is, whether an adapted metric for which an MHD equilibrium is a counter-example to the generalised Grad's conjecture must possess no non-trivial Killing fields. If that were true, it would imply Grad's conjecture, after all, Euclidean space has many Killing symmetries and therefore could not admit a counter-example. However, here we present an example of an MHD equilibrium without Killing symmetries on a manifold that admits Killing fields. In addition, the metric that makes the vector an MHD equilibrium of this example lies in the dense and open set of adapted metrics constructed in Theorem \ref{thm:introalphafixed1}. This shows that this set contains, in general, metrics that do admit non-trivial Killing fields.

\begin{example}\label{eg:MHD-with-Killing-on-M}
Set $M = T^3$ with the Cartesian coordinates $(\zeta, \theta,\phi)$. Define 
\[ X = b(\zeta) \partial_\theta + \iota_0 b(\zeta) \partial_\phi,\quad \alpha = b(\zeta) d\theta + \iota_0 b(\zeta) d\phi,\quad \mu = d\zeta\wedge d\theta \wedge d\phi\] 
for any $b(\zeta) > 0$ and $\iota_0 \in \mathbb{R}\setminus\{0\}$. It can be computed that,
\[ L_X \mu = d i_X \mu = 0. \]
Therefore $X$ is a non-vanishing vector field on $M$ which is volume-preserving with respect to $\mu$.
    
Now, consider the smooth family of symmetric bilinear forms $g_\epsilon$ on $M$ with coefficient matrix $[g_\epsilon]$ in $(\zeta,\theta,\phi)$ given by
\[ [g_\epsilon] = \begin{pmatrix}
    \left(1-\epsilon (\iota_0^{-1}+\iota_0) f(\theta,\phi)\right)^{-1} & 0 & 0 \\
    0 & 1- \iota_0 \epsilon f(\theta,\phi) & \epsilon f(\theta,\phi) \\
    0 & \epsilon f(\theta,\phi) & 1 - \iota_0^{-1} \epsilon f(\theta,\phi)
\end{pmatrix},\]
where $f$ is an arbitrary fixed smooth function and $\epsilon$ is small enough so that $1-\epsilon(\iota_0^{-1}+\iota_0)f(\theta,\phi) > 0$. Observe that when $\epsilon = 0$, the coefficient matrix $[g_0]$ is the identity. This implies that there is an open neighbourhood $V\subset \mathbb{R}$ containing $0$ for which $g_\epsilon$ is positive definite for all $\epsilon \in V$. It follows that $g_\epsilon$ is a Riemannian metric on $M$ for all $\epsilon \in V$.

A calculation shows
\[ i_X g_\epsilon =  \alpha, \qquad \det([g_\epsilon])=1 \]
so that 
\[\mu_{g_\epsilon} = d\zeta \wedge d\theta \wedge d\phi = \mu\] 
for any $\epsilon \in V$. Consequently, as defined in Equation \ref{eq:adaptedsubspace}, $g_\epsilon \in \mathcal{M}_{(X,\alpha,\mu)}$. That is, for all $\epsilon \in V$, $g_\epsilon$ is a metric adapted to $(X,\alpha,\mu)$.

Now, let 
\[p := \tfrac12 \alpha(X) = \tfrac12 (1+\iota_0^2)b(\zeta)^2 = \tfrac12 g_\epsilon(X,X).\]
Then the tuple $(X,\alpha, \mu, p)$ is a guided flow as defined in Definition \ref{def:guidedflow}. Indeed, it has already been shown that $L_X \mu = 0$ and it is clear that $\alpha(X) > 0$ from the assumption that $b(\zeta) > 0$. Furthermore,
\[ dp = (1+ \iota_0^2)b(\zeta) b'(\zeta) d\zeta \implies dp(X) = 0. \]
Finally
\[ d\alpha = b'(\zeta)(d\zeta \wedge d\theta + \iota_0 d\zeta \wedge d\phi) \implies d\alpha \wedge dp = 0. \]
Hence, $(X,\alpha,\mu,p)$ is a guided flow. In addition to being a guided flow, it can be computed that 
\[i_X d\alpha = -dp.\] 
Hence, for any $g \in \mathcal{M}_{(X,\alpha,\mu)}(M)$, $X$ is an MHD equilibrium on $(M,g)$. In particular, $X$ is an MHD equilibrium with respect to the metric $g_\epsilon$ for all $\epsilon \in V$.

Now, observe that $[g_\epsilon]$ is independent of $\zeta$. Consequently,  $\partial_\zeta$ is a Killing field on $(M,g_\epsilon)$ for all $\epsilon \in V$. Yet, 
\begin{align*}
[\partial_\zeta,X] &= b^\prime(\zeta)\partial_\theta + \iota_0 b^\prime(\zeta) \partial_\phi,\\
\partial_\zeta(p) &= dp(\partial_\zeta) = (1+ \iota_0^2)b(\zeta) b'(\zeta),
\end{align*}
proving that if $b(\zeta)$ is not constant, then $\partial_\zeta$ is not a symmetry of $X$ nor is it a symmetry of the pressure function $p=\tfrac12 \alpha(X)$. 

In fact, by a judicious choice of $f(\theta,\phi)$, it can be shown that $X$ and $p$ have no symmetries that are Killing fields whenever $\epsilon \in V\setminus \{0\}$. We will show this by choosing $f$ so that $g_\epsilon$ is in the $C^\infty$-dense and $C^0$-open subset $\cU \subset \cM_{(X,\alpha,\mu)}(M)$ from the proof of Theorem \ref{thm:introalphafixed1} (for some regular level set of $p$, which will exist when $b'(\zeta) \neq 0$). Then it follows that there are no Killing fields $K$ such that $L_K X = 0$ or $K(p) = 0$.

To show $g_\epsilon \in \cU$ it is enough to construct $f$ so that the companion field $Y$ to $X$ has a modulus $\|Y\|_\epsilon^2 := g_\epsilon(Y,Y)$ which, when restricted to a regular level set $S$ of $p$, has a unique local maximum.
As defined in Theorem \ref{thm:prev-paper}, the companion field $Y$ is the unique vector field that satisfies 
\[ i_Y \mu = \frac{1}{\alpha(X)} \alpha \wedge dp. \]
It can be computed that  
\[ \frac{1}{\alpha(X)} \alpha\wedge dp = \tfrac12 b'(\zeta)\left( d\theta \wedge d\zeta + \iota_0 d\phi\wedge d\zeta \right). \]
Setting $Y = Y^\zeta \partial_\zeta + Y^\theta \partial_\theta + Y^\phi \partial_\phi$, it is also seen that 
\[ i_Y \mu = Y^\zeta d\theta \wedge d\phi - Y^\theta d\zeta\wedge d\phi + Y^\phi d\zeta \wedge d\theta.\]
It follows that the companion field $Y$ is given by 
\[ Y = \tfrac12 b^\prime(\zeta) \left( \iota_0 \partial_\theta - \partial_\phi \right). \]
Finally, 
\[ \|Y\|^2_\epsilon = b'(\zeta)^2 (1+\iota_0^2)\left(1 - \epsilon (\iota_0^{-1} + \iota_0)f(\theta,\phi) \right).\] 
Hence, for any regular level set $S$ of $p$ (that is, any set of the form $\{\zeta = \zeta_0\}$, with $b'(\zeta_0)\neq 0 $), provided we choose $f(\theta,\phi)$ so that $1 - \epsilon (\iota_0^{-1} + \iota_0)f(\theta,\phi)$ has a unique local maximum, $g_\epsilon$ will be in $\mathcal{U}$. It can then be concluded from Theorem \ref{thm:introalphafixed1} that there are no Killing fields $K$ such that $L_K X = 0$ or $K(p) = 0$.
\end{example}

\paragraph{\textbf{Quasi-symmetry.}}

In the hopes of finding non-axisymmetric MHD equilibria, one may try to find quasi-symmetric MHD equilibria which are a generalisation of axisymmetric MHD equilibria. A quasi-symmetric MHD equilibrium on an oriented Riemannian 3-manifold $(M,g_0)$ with boundary is an MHD equilibrium $(X,p)$ on $M$ (with $X$ nowhere-vanishing) together with a vector field $u$ (called the quasi-symmetry) such that \cite{B}
\begin{equation*}
L_u \|X\| = 0, \qquad L_u i_X\mu = 0,\qquad L_u X^{\flat} = 0
\end{equation*}
where the operations are taken with respect to a metric $g_0$ and $\mu = \mu_{g_0}$. Interestingly, setting $\alpha = i_X g_0$ these equations for $u$ are equivalent to
\begin{equation*}
L_u \alpha(X) = 0, \qquad L_u i_X\mu = 0, \qquad L_u \alpha = 0.
\end{equation*}
This means that $(X,p)$ remains a quasi-symmetric equilibrium on $(M,g)$ for any adapted metric $g \in \cM_{(X,\alpha,\mu)}(M)$.

In particular, one may take an axisymmetric toroidal domain $M \subset \RR^3$ and consider axisymmetric solutions $(X,p)$ of the MHD equations on $M$ with $X$ nowhere zero. Then, if $K$ denotes a Killing field on $M$ generating the rotational symmetry of $M$, then denoting by $g_0$ the Euclidean metric on $M$ and setting $\alpha = i_X g_0$, we have
\begin{equation*}
L_K \alpha(X) = 0, \qquad L_K i_X\mu = 0, \qquad L_K \alpha = 0.
\end{equation*}
In particular, for any adapted metric $g \in \cM_{(X,\alpha,\mu)}(M)$, the MHD equilibrium is quasi-symmetric. Using our main results, this provides many examples of quasi-symmetric MHD equilibria without Killing symmetries in the setting of abstract metrics.
 
\section{Discussion}\label{sec:discuss}

In this paper, we showed that one may start with an MHD equilibrium $(X,p)$ on some Riemannian manifold and perturb the metric so that the resulting Riemannian manifold still has $(X,p)$ as an MHD equilibrium but without symmetries. This was all done in an abstract setting and the question remains whether this has any application to our statement of Grad's conjecture (Conjecture \ref{conj:Grad} in Section \ref{sec:Grad conjecture}).

One piece of insight this brings into Grad's conjecture is that the setting of abstract metrics is far different from considering the case of the Euclidean metric. In particular, the conjecture is truly one about Euclidean space. In particular, attacking the classical conjecture requires tools that do not work in the abstract setting and instead are tailored for some special properties of Euclidean space.

If the conjecture is false, a potential avenue for producing counter-examples would be to take a solid torus $M$ with an abstract metric $g$ with an MHD equilibrium $(X,p)$ which has no non-trivial Killing symmetries and see if there exists an isometric embedding $\Psi : M \to \RR^3$. For instance, if one starts with an axisymmetric solid torus $M \subset \RR^3$ with the Euclidean metric $g_0$ and an axisymmetric MHD equilibrium $(X,p)$ on $M$, then perhaps it is possible to perturb the metric $g_0$ to an adapted metric $g$ for which $(X,p)$ does not have any Killing symmetries on $(M,g)$ in such a way that there still is an isometric embedding of $(M,g)$ into Euclidean space. 

The isometric embedding approach just mentioned would be very difficult or perhaps impossible to successfully implement even if the conjecture is false. For instance, we have not seen in this paper the impact on curvature after perturbing to a new adapted metric. Nevertheless, perhaps a looser variation of the above strategy could be fruitful. For instance, if one starts with an MHD equilibrium $(X,p)$ on $(M,g_0)$ and $\eta$ is a 1-form such that $i_Xd\eta = -df$, then any metric $g$ adapted to $(X,\alpha+\eta,\mu)$ will have $(X,p+f)$ as an MHD equilibrium. It may also be necessary to adjust $X$ as well.

\section{Data availability}

No datasets were generated or analysed during the current study.

\section{Competing interests}

The authors have no competing interests to declare that are relevant to the content of this article.

\appendix

\section{On $\partial$-regular values}\label{app:partialregvalues}

Fix manifolds $M$ and $N$ with (possibly empty) boundary and a smooth map $F : M \to N$ between them.

\begin{definition}
An element $y \in N$ is said to be $\partial$-regular if $y$ is a regular value of both $F$ and $F|_{\partial M}$.
\end{definition}

It is well-known that this situation gives a pre-image theorem for $\partial$-regular values.

\begin{theorem}[Pre-image Theorem for $\partial$-regular values]\label{thm:preimagepartial}
If $y \in N$ is a $\partial$-regular value of $F$, then $F^{-1}(y)$ is an embedded submanifold boundary of dimension $\dim M - \dim N$ satisfying
\begin{equation*}
\partial F^{-1}(y) = F^{-1}(y) \cap \partial M. 
\end{equation*}
\end{theorem}

Besides this pre-image theorem, we also note the following corollary of Sard's Theorem.

\begin{corollary}[Sard's Theorem for $\partial$-critical values]\label{Sardforpartial}
Let $C \subset N$ denote the $\partial$-critical values of $F$. Then $C$ has measure zero in $N$.
\end{corollary}

\begin{proof}
Denoting by $C_1$ the critical values of $F$ and $C_2$ the critical values of $F|_{\partial M}$, then we have the equality
\begin{equation*}
C = C_1 \cup C_2.    
\end{equation*}
Both $C_1$ and $C_2$ are measure-zero in $N$ by Sard's Theorem \cite[Thm.~6.10]{L}. Thus, $C$ has zero measure in $N$.
\end{proof}

We will use a result about the density of pre-images of $\partial$-regular values, which is a corollary of Corollary \ref{Sardforpartial}. To state it, we set
\begin{align*}
\text{Reg}_M(F) &= \{x \in M : x \text{ is a regular point of $F$}\},\\
\text{Reg}_N^{\partial}(F) &= \{y \in N : y \text{ is a $\partial$-regular value of $F$}\}.
\end{align*}
The corollary may now be stated as follows.

\begin{corollary}\label{cor:partialdensity}
The set $A = F^{-1}(\operatorname{Reg}^{\partial}_N(F))$ is a dense subset of the open set $U = \operatorname{Reg}_M(F)$.
\end{corollary}

\begin{proof}
The set $U = \text{Reg}_M(F)$ being open is standard (see \cite[Proposition 4.1]{L}). Moreover, the inclusion $A \subset U$ is clear. Now, let $x \in U$ and $V$ a neighbourhood of $x$ in $U$. Because $\Int M$ is dense in $M$ and $V$ is open in $M$, $V_0 = V \cap \Int M$ is a non-empty open subset of $V$. Regarding $V_0$ as an open submanifold of $\Int M$, $F|_{V_0} : V_0 \to N$ is a submersion. Moreover, because $V_0$ has no boundary and $F$ is a submersion, it follows that for every $y \in F(V_0)$ and $v \in T_y N$, there exists $\epsilon > 0$ and a curve $\gamma : (-\epsilon,\epsilon) \to N$ such that $\gamma(0)= y$ and $v = \gamma'(0)$. In particular, $F(V_0) \subset \Int N$. Thus, the image $\Imag F|_{V_0} = F(V_0)$ is an open subset of $\Int N$ (see \cite[Prop.~4.28]{L}). In particular, $F(V_0) \subset N$ is not a set of measure zero. Thus, neither is $F(V) \supset F(V_0)$. However, this means by Corollary \ref{Sardforpartial} that $F(V) \cap \text{Reg}^{\partial}_N(F) \neq \emptyset$. We conclude that $V \cap A \neq \emptyset$ and hence that $A$ is dense in $U$.
\end{proof}

\section{Some geometry on surfaces and boundaries}\label{app:surfacegeo}

The main purpose of this section is to prove Lemma \ref{lem:vanishingonsurface} which we recall here.

\vanishingonsurface*

The first part of the lemma is straight forward.

\begin{proof}[Proof of Lemma \ref{lem:vanishingonsurface} (i)]
It is well-known that $K_S$ is indeed a smooth vector field. Moreover, by naturality of Lie derivatives of tensor fields,
\begin{equation*}
L_{K_S}i^*g = i^*L_K g = 0.
\end{equation*}
\end{proof}

To prove part (ii) of Lemma \ref{lem:vanishingonsurface}, the general case requires a discussion of geodesics in the case of manifolds with boundary in a very particular case. For this, we introduce some theory which is known in the literature but whose proofs are not easily found. To focus on what is essential, we make the following temporary definition. 

\begin{definition}
Let $(M,g)$ be a Riemannian manifold with boundary. Let $n : \partial M \to TM$ denote the inward normal on $M$. Then, if $\epsilon > 0$, a smooth curve $\gamma : [0,\epsilon) \to M$ is said to be an inward geodesic at $x \in \partial M$ if
\begin{enumerate}
    \item $\gamma(0) = x$ and $\gamma'(0) = n|_x$,
    \item $\gamma((0,\epsilon)) \subset \Int M$, and
    \item the restricted curve $(0,\epsilon) \to \Int M$ is a geodesic on the Riemannian manifold $(\Int M,\imath^*g)$ in the usual sense, where $\imath : \Int M \subset M$ is the interior inclusion.
\end{enumerate}
\end{definition}

For completeness, we will prove the following about inward geodesics.

\begin{proposition}[Riemannian normal-collars]\label{prop:Riemannian normal-collars}
Let $(M,g)$ be a Riemannian manifold with boundary. Let $n : \partial M \to TM$ denote the inward normal on $M$. Then, there exists a smooth positive function $0 < \delta \in C^{\infty}(\partial M)$ and a map $\varphi : U_\delta \to M$ which is a diffeomorphism onto its open image $\Sigma(U_\delta)$ where
\begin{equation*}
U_\delta = \{(x,t) \in \partial M \times [0,\infty) : 0 \leq t < \delta(x) \}
\end{equation*}
and for $x \in \partial M$, the curve
\begin{equation*}
[0,\delta(x)) \ni t \mapsto \Sigma(x,t)    
\end{equation*}
is an inward geodesic at $x$.
\end{proposition}

\begin{remark}\label{rmk:extending-the-metric}
If $(M,g)$ is a Riemannian manifold with boundary, then it follows from the existence of the boundaryless double of $M$ (see \cite[Example 9.32]{L}) that $M$ may be embedded as a regular domain in a manifold $\tilde{M}$ without boundary. Moreover, as is standard, one may extend the metric $g$ to a metric $\tilde{g}$ in $\tilde{M}$. This is done by first: taking slice charts at boundary points $x \in \partial M$, extending the smooth coefficients metric in slice charts with the fact that positive definiteness of symmetric matrices (of given size) is a stable property and second: using a suitable partition of unity in $\tilde{M}$ choosing arbitrary local metrics in $\tilde{M} \backslash M$. 
\end{remark}

To prove Proposition \ref{prop:Riemannian normal-collars}, we first check existence and uniqueness of inward geodesics.

\begin{proposition}\label{prop:uniqueness-of-inward-geodesics}
Let $(M,g)$ be a Riemannian manifold with boundary and $x \in \partial M$. Then, there exists an inward geodesic at $x$. Let $0 < \epsilon_1 < \epsilon_2$ and $\gamma_i : [0,\epsilon_i) \to M$ for $i \in \{1,2\}$ be inward geodesics at $x$. Then, on $\gamma_1 = \gamma_2$ on $[0,\epsilon_1)$.
\end{proposition}

\begin{proof}
Following Remark \ref{rmk:extending-the-metric}, we may fix a Riemannian manifold $(\tilde{M},\tilde{g})$ whereby $M$ is contained in $\tilde{M}$ as a regular domain of $M$. We make use of various elementary facts about geodesics defined on open intervals in Riemannian manifolds without boundary. Let $\gamma : [0,\epsilon) \to \tilde{M}$ be an inward geodesic at a point $x \in \partial M$ viewed as a curve in $\tilde{M}$. We claim that the unique maximal geodesic $\tilde{\gamma}: I \to \tilde{M}$ (where $I$ is an open interval about $0$) satisfying $\tilde{\gamma}'(0) = \gamma'(0)$ extends the curve $\gamma$.

To prove the claim, consider the point $V = \gamma'(0) \in T\tilde{M}$. Observe that there exists $\epsilon > 0$ and a neighbourhood $U$ of $V$ in $T\tilde{M}$ such that for each $W \in U$, there a geodesic $\kappa : (-\epsilon,\epsilon) \to \tilde{M}$ satisfying $\kappa'(0) = W$. Moreover, because $\gamma' : [0,\epsilon) \to T\tilde{M}$ is continuous, there exists $\delta > 0$ such that $\gamma'([0,\delta)) \subset U$. Thus, for each $t \in (0,\delta)$, there exists a geodesic $\kappa_t : (-\epsilon,\epsilon) \to \tilde{M}$ such that
\begin{equation*}
\kappa_t'(0) = \gamma'(t).    
\end{equation*}
By uniqueness of geodesics, we have that for $t \in (0,\delta)$ and $s \in (0,b) \cap (t-\epsilon,t+\epsilon)$,
\begin{equation*}
\gamma(s) = \kappa_t(s-t).
\end{equation*}
Hence, the curve $\kappa : (-\epsilon/2,b) \to \tilde{M}$ given by $\kappa(t) = \gamma(t)$ for $t \in [0,b)$ and 
\begin{equation*}
\kappa(t) = \kappa_{\epsilon/2}(s-\epsilon/2)    
\end{equation*}
is smooth on $\tilde{M}$ and satisfies the geodesic equation on $(-\epsilon/2,0)\cup (0,b)$. Thus, $\kappa$ is a geodesic on $\tilde{M}$. Moreover, 
\begin{equation*}
\kappa'(0) = \gamma'(0).    
\end{equation*}
The claim then follows. 

Hence, by uniqueness of geodesics defined on open intervals in $\tilde{M}$, the Proposition follows from the above claim.
\end{proof}

We now prove Proposition \ref{prop:Riemannian normal-collars}.

\begin{proof}[Proof of Proposition \ref{prop:Riemannian normal-collars}]
Following Remark \ref{rmk:extending-the-metric}, we may fix a Riemannian manifold $(\tilde{M},\tilde{g})$ whereby $M$ is contained in $\tilde{M}$ as a regular domain of $M$. 

Fix a unit normal $n : \partial M \to T\tilde{M}$ on $\partial M$ pointing inward on $M$ and consider the smooth map $H : \partial M \times \RR \to TM$ given by $H(x,t) = t n|_p$. On the other hand, consider the exponential map $\exp : O \to M$ on $(M,g)$ following the notation of \cite[Section 5.5.1]{P}. That is, $v \in O$ if and only if there exists a geodesic $\gamma : I \to M$ with interval $I$ containing $[0,1]$ such that $\gamma'(0) = v$; and by definition, $\exp(v) = \gamma(1)$. Clearly $O$ contains the zero section of $TM$ and for every $v \in O$, $t v \in O$ for $t \in [0,1]$. 

In particular, the set $V = H^{-1}(O)$ is an open neighbourhood of $\{0\} \times \partial M$ in $\partial M \times \RR$ for which there exist open intervals $I_x$ for each $x \in S$ containing zero such that $V = \{(x,t) \in S \times t : t \in I_x\}$. Denote by $h : V \to O$ the according restriction of $H$. We obtain a smooth map $\tilde{\sigma} = \exp \circ h$. Then, it is easy to see by the definition of exponential map that for $(x,t) \in V$,
\begin{equation*}
\tilde{\sigma}(x,0) = x, \qquad \frac{\partial \tilde{\sigma}}{\partial t}\bigg|_{(x,t)} = n|_x
\end{equation*}
where for $(t,x) \in V$, $\frac{\partial \tilde{\sigma}}{\partial t}(x,t) = \gamma'(t)$ where $\gamma$ is the curve $I_x \ni t \mapsto \tilde{\sigma}(x,t)$. From this, it is clear that the tangent map $T_{(x,t)}\tilde{\sigma}$ is surjective and thus invertible by the dimension of the domain and codomain of $h$ being equal. Thus $\tilde{\sigma}$ is a local diffeomorphism.

By the existence of defining functions (see \cite[Theorem 5.48]{L}), there exists a smooth function $f \in C^{\infty}(\tilde{M})$ such that $f^{-1}([0,\infty)) = M$ is a regular sublevel set of $f$. That is, $f^{-1}(0)$ is a regular level set in $\tilde{M}$ which necessarily coincides with $\partial M$. Indeed, this is because $f^{-1}(0,\infty) \subset \Int M$ and $f^{-1}(0) \cup f^{-1}((0,\infty)) = M$, showing that $\partial M \subset f^{-1}(0)$, and because $0$ is a regular value and the minimum of $f$ on $M$, we must also have $f^{-1}(0) \subset \partial M$. Because $n$ is inward-pointing on $M$,
\begin{equation*}
V' = \{(x,t) \in V : df|_x\left(\frac{\partial \sigma}{\partial t}(x,t)\right) > 0\}   
\end{equation*}
is an open neighbourhood of $\partial M \times \{0\}$ in $\partial M \times \RR$. 

Then, for $(x,t) \in U$ the following holds:
\begin{enumerate}
    \item If $t < 0$, $\sigma(x,t) \in \tilde{M} \backslash M$
    \item If $t \geq 0$, $\sigma(x,t) \in M$
\end{enumerate}
Hence, the map $\sigma : V'_{\geq 0} \to M$ whereby
\begin{equation*}
V'_{\geq 0} = V' \cap (\partial M \times [0,\infty)), \qquad \sigma(x,t) = \tilde{\sigma}(x,t),
\end{equation*}
is a local diffeomorphism. Hence, for each $x \in \partial M$, we may choose a product neighbourhood of $(x,0) \in W_x \times [0,\epsilon_x)$ in $V'_{\geq 0}$ where $\epsilon_x > 0$ such that $\sigma$ is a diffeomorphism onto its image $U_x$. In particular, for each $x \in \partial M$, there exists a vector field $X_x : U_x \to TU_x$ on $U_x$ such that for each $y \in \partial U_x$, there is an inward geodesic $\gamma : [0,\epsilon_x) \to U_x$ at $y$ which satisfies $X_x(\gamma(t)) = \gamma'(t)$. It follows by the uniqueness of inward geodesics on $M$ (Proposition \ref{prop:uniqueness-of-inward-geodesics}) that we obtain a neighbourhood $U$ of $\partial M$ in $M$ and a vector field $X : U \to TU$ such that for each $y \in \partial U$, there is an inward geodesic $\gamma : [0,\epsilon) \to U$ at $y$ which satisfies $X(\gamma(t)) = \gamma'(t)$. The result now follows from applying the Boundary Flowout Theorem.
\end{proof}

We now move on to prove part (ii) of Lemma \ref{lem:vanishingonsurface}. For this, we consider a special case which is sufficient for the full result.

\begin{lemma}\label{lem:sufficient-case-vanishingonsurface}
Let $(M,g)$ be a Riemannian manifold with non-empty boundary. Suppose that $K$ is a Killing field on $M$ such that $K|_{\partial M} = 0$. Then, $K = 0$.
\end{lemma}

\begin{proof}
First, because $K|_{\partial M} = 0$, $K$ is in particular tangent to $\partial M$. Thus, we may consider the maximal flow $\Psi^K : \cD \to M$ of $K$ where $\cD \subset \Rone \times M$ denotes the (maximal) domain of the flow of the Killing field. Following the notation in \cite{L}, we set $M_t = \{x \in M : (t,x) \in \cD\}$ and $I_x = \{t \in \Rone : (t,x) \in \cD \}$. Since $K|_\partial M = 0$, for $x \in \partial M$ we have that the flow fixes $x$ and $I_x = \RR$ for all $x \in \partial M$. So $\partial M \subset M_t$ for each $t \in \RR$. For $t \in \RR$, let $g_t$ be the metric so that $(M_t,g_t)$ is the Riemannian manifold structure inherited from $(M,g)$. 

Now, for $t \in \RR$ let $\cG_t$ denote the set of maximal inward geodesic at $x$ in $M_t$ and consider the set
\begin{equation*}
O_{t} = \cup_{\gamma \in \cG_t} \Imag{\gamma}.   
\end{equation*}
On the one hand, letting $T \in \RR$ and $[0,T] = \{sT : s \in [0,1]\}$, for $t \in [0,T]$, considering the inclusion $i_{tT} : M_T \subset M_t$, if $\gamma$ is an inward geodesic at a point $x \in \partial M_t$, then $i_{tT} \circ \gamma$ an inward geodesic at a point $x \in \partial M_T$. Hence,
\begin{equation*}
O_{T} \subset O_t.
\end{equation*}

On the other hand, for $t \in \RR$, we have that $\Psi^K_t : M_t \to M_{-t}$ is an isometry and thus
\begin{equation*}
T \Psi^K_t \circ n_t = n_{-t}
\end{equation*}
and because $\Psi^K_t(\Int M_t) = \Int M_{-t}$, it follows that
\begin{equation*}
\cG_t \ni \gamma \mapsto \Psi^K_t \circ \gamma \in \cG_{-t}  
\end{equation*}
is a bijection. On the other hand, for each $\gamma \in \cG_t$, we may consider the inward geodesics
\begin{equation*}
i_{0t} \circ \gamma, \qquad i_{-t} \circ \Psi^K_t \circ \gamma 
\end{equation*}
on $M$ which must coincide by Proposition \ref{prop:uniqueness-of-inward-geodesics}. Hence, for $x \in O_t$,
\begin{equation*}
\Psi^K_t(x) = x.    
\end{equation*}
Thus, fixing $T > 0$, we have that for $x \in O_T$ that $\Psi^K_T(x) = x$. Moreover, for $t \in [0,T]$, if $x \in O_t$, we also have $\Psi^K_T(x) = x$ for all $x \in O_t$ but also that $O_T \subset O_t$ by the above. Altogether, we have that
\begin{equation*}
\Psi^K_t(x) = x \text{ for all } x \in O_T \text{ and } t \in [0,T].    
\end{equation*}
Hence, we deduce that
\begin{equation*}
K|_{O_T} = 0.
\end{equation*}
On the other hand, by Proposition \ref{prop:Riemannian normal-collars}, the subset $O_T \subset M_T$ contains an open set covered by inward geodesics. Hence, $K$ vanishes on an open subset of $M$. Thus, $K$ vanishes on a non-empty subset of $\Int M$. The vector field $\tilde{K}$ on $\Int M$ related to $M$ is thus a Killing field vanishing on an open subset. Thus, by \cite[Proposition 8.1.4]{P} it follows that $\tilde{K} = 0$ and thus that $K = 0$.
\end{proof}

We now prove Lemma \ref{lem:vanishingonsurface}.

\begin{proof}[Proof of Lemma \ref{lem:vanishingonsurface}]
First consider $\Int S$. There are two cases:
\begin{enumerate}
    \item either $\Int S \subset \partial M$ 
    \item or $\Int S \cap \Int M \neq \emptyset$.
\end{enumerate}

In case (i), since $\partial M$ is embedded in $M$ and $\Int S$ has codimension-1 in $M$, $\Int S$ (by the Inverse Function Theorem) is an open submanifold of $\partial M$. In particular, taking a boundary chart $(U,\varphi)$ of $M$ such that $\partial U \subset \Int S$, we have that $U$ is a codimension zero submanifold of $M$ with boundary such that $K|_{\partial U} = 0$. 

In case (ii), because $\Int S$ is immersed in $M$, $\Int S$ is locally embedded in $M$. Hence, we may consider an open subset $V$ of $\Int S$, and a slice chart $(\tilde{U},\varphi)$ of $V$ in $\Int M$. Then, setting $U = \varphi^{-1}(\tilde{U} \cap \HH^n)$, we have that $U$ is a codimension 0 submanifold of $M$ with boundary such that $K|_{\partial U} = 0$.

In either case, we conclude from Lemma \ref{lem:sufficient-case-vanishingonsurface} that $K$ vanishes on an open subset of $\Int M$. Thus, we may conclude in the same way as in the proof of Lemma \ref{lem:sufficient-case-vanishingonsurface} that $K = 0$ on $M$.
\end{proof}

\bibliographystyle{plain}
\bibliography{paper.bib}

\begin{thebibliography}{10}

\bibitem{A}
Vladimir~I. Arnold.
\newblock Sur la topologie des ecoulements stationaries des fluides parfaits.
\newblock {\em CR Acad. Sci Paris A}, 261:17--20, 1965.

\bibitem{AK}
Vladimir~I. Arnold and Boris~A. Khesin.
\newblock {\em Topological Methods in Hydrodynamics}, volume 125 of {\em
  Applied Mathematical Sciences}.
\newblock Springer, 1998.

\bibitem{BL}
Oscar~P. Bruno and Peter Laurence.
\newblock Existence of three-dimensional toroidal mhd equilibria with
  nonconstant pressure.
\newblock {\em Communications on Pure and Applied Mathematics},
  49(7):717–764, July 1996.

\bibitem{B}
Joshua~W. Burby, Nikos Kallinikos, and Robert~S. MacKay.
\newblock Some mathematics for quasi-symmetry.
\newblock {\em Journal of Mathematical Physics}, 61(9):093503, 2020.

\bibitem{C}
Robert Cardona.
\newblock The topology of {B}ott integrable fluids.
\newblock {\em Discrete and Continuous Dynamical Systems}, 42(9):4321--4345,
  2022.

\bibitem{Ce}
Antoine~J. Cerfon and Jeffrey~P. Freidberg.
\newblock ``{{One}} size fits all'' analytic solutions to the
  {{Grad}}{\textendash}{{Shafranov}} equation.
\newblock {\em Physics of Plasmas}, 17(3):032502, 2010.

\bibitem{CH}
Shiing-Shen Chern and Richard~S. Hamilton.
\newblock On {R}iemannian metrics adapted to three-dimensional contact
  manifolds.
\newblock In Friedrich Hirzebruch, Joachim Schwermer, and Silke Suter, editors,
  {\em Arbeitstagung Bonn 1984}, page 279–308, Berlin, Heidelberg, 1985.
  Springer Berlin Heidelberg.

\bibitem{CDG2}
Peter Constantin, Theodore~D. Drivas, and Daniel Ginsberg.
\newblock Flexibility and {R}igidity in {S}teady {F}luid {M}otion.
\newblock {\em Communications in Mathematical Physics}, 385(1):521–563, 2021.

\bibitem{ELP}
Alberto Enciso, Alejandro Luque, and Daniel Peralta-Salas.
\newblock {MHD} equilibria with nonconstant pressure in nondegenerate toroidal
  domains.
\newblock {\em Journal of the European Mathematical Society}, 2023.
\newblock Published online first.

\bibitem{EG}
John Etnyre and Robert Ghrist.
\newblock Contact topology and hydrodynamics: {{I}}. {{Beltrami}} fields and
  the {{Seifert}} conjecture.
\newblock {\em Nonlinearity}, 13(2):441--458, 2000.

\bibitem{Ger}
Wadim Gerner.
\newblock Typical field lines of {B}eltrami flows and boundary field line
  behaviour of {B}eltrami flows on simply connected, compact, smooth manifolds
  with boundary.
\newblock {\em Annals of Global Analysis and Geometry}, 60(1):65--82, 2021.

\bibitem{G}
Harold Grad.
\newblock Toroidal {{Containment}} of a {{Plasma}}.
\newblock {\em The Physics of Fluids}, 10(1):137--154, 1967.

\bibitem{GP}
Victor Guillemin and Alan Pollack.
\newblock {\em Differential {{Topology}}}, volume 370 of {\em {{AMS Chelsea
  Publishing}}}.
\newblock {American Mathematical Society}, 1974.

\bibitem{H}
Per Helander.
\newblock Theory of plasma confinement in non-axisymmetric magnetic fields.
\newblock {\em Reports on Progress in Physics}, 77(8):087001, 2014.

\bibitem{KKPS}
Boris Khesin, Sergei Kuksin, and Daniel Peralta-Salas.
\newblock {KAM} theory and the {3D} {E}uler equation.
\newblock {\em Advances in Mathematics}, 267:498--522, 2014.

\bibitem{L}
John~M. Lee.
\newblock {\em Introduction to {{Smooth Manifolds}}}, volume 218 of {\em
  Graduate {{Texts}} in {{Mathematics}}}.
\newblock {Springer}, 2012.

\bibitem{Lima}
Elon~L. Lima.
\newblock The {{Jordan-Brouwer Separation Theorem}} for {{Smooth
  Hypersurfaces}}.
\newblock {\em The American Mathematical Monthly}, 95(1):39--42, 1988.

\bibitem{Lo}
Dietrich Lortz.
\newblock {{\"U}ber die {E}xistenz toroidaler magnetohydrostatischer
  {G}leichgewichte ohne rotationstransformation}.
\newblock {\em Zeitschrift f{\"u}r angewandte Mathematik und Physik ZAMP},
  21(2):196--211, 1970.

\bibitem{Per}
Daniel Peralta-Salas.
\newblock Selected topics on the topology of ideal fluid flows.
\newblock {\em International Journal of Geometric Methods in Modern Physics},
  13(Supp. 1):1630012, 2016.

\bibitem{PRT}
Daniel {Peralta-Salas}, Ana Rechtman, and Francisco Torres~De Lizaur.
\newblock A characterization of {{3D}} steady {{Euler}} flows using commuting
  zero-flux homologies.
\newblock {\em Ergodic Theory and Dynamical Systems}, 41(7):2166--2181, 2021.

\bibitem{PDP}
David Perrella, Nathan Duignan, and David Pfefferl{\'e}.
\newblock Existence of global symmetries of divergence-free fields with first
  integrals.
\newblock {\em Journal of Mathematical Physics}, 64(5):052705, 2023.

\bibitem{PPS}
David Perrella, David Pfefferl{\'e}, and Luchezar Stoyanov.
\newblock Rectifiability of divergence-free fields along invariant 2-tori.
\newblock {\em Partial Differential Equations and Applications}, 3(4):50, 2022.

\bibitem{P}
Peter Petersen.
\newblock {\em Riemannian {{Geometry}}}, volume 171 of {\em Graduate {{Texts}}
  in {{Mathematics}}}.
\newblock {Springer New York}, 2006.

\bibitem{S}
G{\"u}nter Schwarz.
\newblock {\em Hodge {{Decomposition}}\textemdash{{A Method}} for {{Solving
  Boundary Value Problems}}}, volume 1607 of {\em Lecture {{Notes}} in
  {{Mathematics}}}.
\newblock {Springer}, 1995.

\bibitem{Sul}
Dennis Sullivan.
\newblock A foliation of geodesics is characterized by having no “tangent
  homologies”.
\newblock {\em Journal of Pure and Applied Algebra}, 13(1):101--104, 1978.

\bibitem{W}
Edward Witten.
\newblock Supersymmetry and {M}orse theory.
\newblock {\em Journal of Differential Geometry}, 17(4):661--692, 1982.

\end{thebibliography}

\end{document}